\documentclass{alggeom} 
\usepackage[utf8]{inputenc}

\usepackage[colorlinks=true,pagebackref,hyperindex]{hyperref}
\usepackage{mathrsfs} 
\usepackage[all]{xy}
\usepackage{color}
\usepackage{amsfonts}
\usepackage{amscd}
\usepackage{amssymb,latexsym,amsmath,amscd, graphicx} 
\usepackage[makeroom]{cancel}
\usepackage{bbm} 
\usepackage[normalem]{ulem}
\usepackage{xcolor}

\definecolor{darkgreen}{rgb}{0.0, 0.7, 0.0}

\definecolor{cyan}{cmyk}{1,0,0,0}

\newcommand{\bdg}{\begin{dg}}
\newcommand{\edg}{\end{dg}}

\newtheorem{tm}{Theorem}[section]
\newtheorem{lm}[tm]{Lemma}
\newtheorem{pr}[tm]{Proposition}
\newtheorem{rmk}[tm]{Remark}

\newtheorem{fact}[tm]{Fact}
\newtheorem{??}[tm]{Question}

\newcommand{\ben}{\begin{enumerate}}
\newcommand{\een}{\end{enumerate}}
\newcommand{\bit}{\begin{itemize}}
\newcommand{\eit}{\end{itemize}}
\newcommand{\beq}{\begin{equation}}
\newcommand{\eeq}{\end{equation}}
\newcommand{\la}{\label}
\newcommand{\n}{\noindent}
\newcommand\ci{\cite}

\font\tenmsb=msbm10
\font\sevenmsb=msbm7
\font\fivemsb=msbm5
\newfam\msbfam
\textfont\msbfam=\tenmsb
\scriptfont\msbfam=\sevenmsb
\scriptscriptfont\msbfam=\fivemsb
\def\Bbb#1{{\fam\msbfam #1}}
\font\teneufm=eufm10
\font\seveneufm=eufm7
\font\fiveeufm=eufm5
\newfam\eufmfam
\textfont\eufmfam=\teneufm
\scriptfont\eufmfam=\seveneufm
\scriptscriptfont\eufmfam=\fiveeufm
\def\frak#1{{\fam\eufmfam\relax#1}}

\newcommand{\im}{ \hbox{\rm Im} }

\newcommand\oql{\overline{\Bbb Q}_\ell}
\newcommand\comp{{\Bbb C}}

\newcommand\zed{{\Bbb Z}}

\newcommand\s{\sigma}

\newcommand{\w}[1]{\widetilde{#1}}

\newcommand{\ov}[1]{\overline{#1}}

\newcommand{\m}[1]{\mathcal{#1}}
\newcommand{\ms}[1]{\mathscr{#1}}

\newcommand{\ptd}[1]{ \,^{\frak p}\!\tau_{ \leq {#1} } }

\newcommand{\pcs}{ \,^{\frak p}\!{\mathcal H}   }

\newcommand{\field}{k} 
\newcommand{\Gm}{\mathbb G_m}
\newcommand{\FC}{C^{(1)}} 

\newcommand{\ra}{\rightarrow} 
\newcommand{\xra}{\xrightarrow}
\newcommand{\ah}{\alpha}
\newcommand{\gam}{\gamma}
\newcommand{\Aa}{\mathbb{A}}

\newcommand{\brg}{J}
\newcommand{\base}{B}
\newcommand{\csub}{C/\base}
\newcommand{\csubf}{C^{(\base)}/\base}

\begin{document}

\title{A cohomological Non Abelian Hodge Theorem \\ in positive characteristic}
\author{Mark Andrea  de Cataldo}
\email{mark.decataldo@stonybrook.edu}
\address{Mathematics Department\\ Stony Brook University\\ Stony Brook NY, 11794-3651, USA}
\author{Siqing Zhang}
\email{siqing.zhang@stonybrook.edu}
\address{Mathematics Department\\ Stony Brook University\\ Stony Brook NY, 11794-3651, USA}

\classification{14D20}
\keywords{Higgs Bundles, Perverse Filtrations}

\begin{abstract}
We start with a curve over an algebraically closed  ground field of positive characteristic $p>0$. 
By using specialization in cohomology techniques,  under suitable natural coprimality conditions, we prove a cohomological Simpson Correspondence between the moduli space of Higgs bundles and the one of  connections
on the  curve. 
We also prove a new $p$-multiplicative periodicity concerning the cohomology rings of Dolbeault moduli spaces of  degrees differing by a factor of $p$.
By coupling this $p$-periodicity in characteristic $p$ with lifting/specialization techniques in mixed characteristic,
we  find, in arbitrary characteristic,  cohomology ring isomorphisms between the cohomology rings of Dolbeault moduli spaces for  different degrees coprime to the rank.
It is interesting that this last  result is  proved as follows: we  prove a weaker version in positive characteristic; we  lift  and strengthen the weaker version  to the result in characteristic zero; 
 finally, we specialize the  result to positive characteristic.
 The moduli spaces we work with admit certain natural morphisms (Hitchin, de Rham-Hitchin, Hodge-Hitchin), and all the cohomology ring isomorphisms we find are filtered isomorphisms for the resulting perverse Leray filtrations.
 \end{abstract}

\maketitle

\tableofcontents

\section{Introduction}\la{intro}$\;$

Let $C$ be a connected  projective nonsingular curve over the complex numbers. 
The Non Abelian Hodge Theorem (a.k.a. the Simpson Correspondence) (\ci{si mo I, si mo II})   establishes that three
rather different moduli spaces are canonically homeomorphic to each other: the de Rham moduli space $M_{dR}$ of rank $r$  connections on $C$; the Dolbeault  moduli space $M_{Dol}$  of rank $r$
and degree zero Higgs bundles on $C$;  
the Betti moduli space $M_B$  of representations of the fundamental group
of $C$ into $GL(r, \comp)$. 
There is also the Hodge moduli space $M_{Hod}$ of $t$-connections (\ci{si naf}) that in some sense subsumes $M_{Dol}$ and $M_{dR}$.
  For the variant concerning  nonsingular moduli for bundles of  (non zero) degree coprime to the rank,  see \ci{ha-th-PLMS}.
For a brief summary concerning the Hodge, Dolbeault and de Rham moduli spaces, see \S\ref{rino}.

In this paper, we also work over an algebraically closed ground field of positive characteristic, where,
 even though many beautiful results are  available, the situation is less clear.
 Since there seems to be no Betti picture that fits well with a possible Simpson Correspondence,
 in this paper,
 by Simpson Correspondence in characteristic $p>0$, we  mean some kind of relation between Higgs bundles (Dolbeault picture) and connections (de Rham picture).

\ci[\S4]{og-vo} establishes, among other things,   a   Simpson  Correspondence  between the stack of Higgs bundles with  nilpotent Higgs field for the Frobenius twist $C^{(1)}$ of the curve $C$,   
and the stack of  connections on the curve $C$ with nilpotent $p$-curvature tensor.
\ci[Thm.\;3.29,\;Lm.\;3.46]{gr-2016} proves that there is a pair of morphisms  $M_{Dol} (\FC) \to A(\FC) \leftarrow M_{dR}(C)$ which  are \'etale locally equivalent over the 
Hitchin base $A(\FC)$  (\S\ref{rino}), both for the coarse moduli spaces, as well as for the stacks.   \ci[Thm. 1.2]{ch-zh} proves an  analogous result at the stack level, for arbitrary reductive groups in place of the general linear group.
The reader can also consult \ci{la-sh-zu} for generalizations of the isomorphism in \ci{og-vo} to the study of Higgs-de Rham flows for schemes in positive and mixed characteristic.
One  recovers the aforementioned nilpotent Simpson Correspondence in 
characteristic $p>0$ in \ci{og-vo}, by taking  the fibers of the pair of morphisms  over the origin in $A(\FC)$. More generally, we get 
a kind of Simpson Correspondence:  for every closed  point in $A(\FC)$, the two fibers
of the morphisms $M_{Dol}(\FC) \to A(\FC) \leftarrow M_{dR}(C)$ are non-canonically isomorphic varieties, and thus have isomorphic \'etale cohomology rings. Note that these results relate Higgs bundles
of degree $d$  on $\FC$ to connections of degree $dp$ on $C$.

None of these results seems to imply  a global statement concerning (the cohomology) of the Dolbeault and  of the de Rham moduli spaces.
In short,  it seems that we are still  missing a (cohomological)  global Simpson Correspondence  in positive characteristic.

In this paper, we prove such  a new cohomological Simpson Correspondence result for curves over an algebraically closed field of  positive characteristic $p>0$,  as well as a series of new allied results in arbitrary characteristics.
The methods we use center on the use of vanishing cycles and of the specialization morphism in equal and in mixed characteristic.
In order to use these techniques, we need to establish the smoothness of certain morphisms and the properness of certain other morphisms.
Once this is done, we still need to come to terms with the fact that the specialization morphisms may fail 
to be defined, because the moduli spaces we work with are not proper over the ground field. While this issue is circumvented in the proofs of the results in \S\ref{the mainz}, it is not
in the proofs of the results in \S\ref{omq2}, where we use the compactification results of \ci{de-zh B}, and their application to specialization morphisms.  

Let us describe the main results of this paper. 
First of all, all the cohomology rings we deal with carry natural filtrations, called perverse Leray filtration, associated with the various
morphisms  --Hitchin, de Rham-Hitchin, Hodge-Hitchin (\S\ref{rino})-- exiting these moduli space. In what follows we omit these filtrations from the
notation. 

Let $C/\field$ be a nonsingular connected projective curve over an algebraically closed field
of characteristic $p\geq 0.$ Let $\ell$ be  a prime, invertible in the ground field. Since the rank is fixed in what follows, we drop it from the notation.

{\bf Theorem  \ref{zmain} (Cohomological Simpson Correspondence $char(\field)=p>0$, I)  and its refinement Theorem \ref{via pb}
(Cohomological Simpson Correspondence $char(\field)=p>0$, II).} 
Let $p>0.$ We work under  natural assumptions on  the rank $r$ and  degree $d$  of the vector bundles
involved,  and on the characteristic $p$:   namely, $d=\overline{d}p$ is a multiple of the characteristic, and    
the g.c.d. $(r,d)=1$. Note that then $(r,p)=1.$ The first condition
is to have non-empty  de Rham space/stack; the second one   is to have nonsingular moduli spaces.
 Then we prove that there is a canonical  filtered isomorphism   between the corresponding \'etale cohomology rings
\beq\la{eq mt}
H^* \left(M_{Dol}(C; d), \oql\right) \simeq H^*\left(M_{dR}(C;d), \oql\right).
\eeq
Unlike \ci{og-vo, ch-zh, gr-2016},    (\ref{eq mt}) relates the \'etale cohomology rings  of the Dolbeault and de Rham moduli spaces,  for the \underline{same} curve $C$ and the \underline{same} degree. 
While the  Frobenius twist $\FC$ does not appear in the statement of (\ref{eq mt}), it  plays a key role in the proof.

 {\bf Theorem \ref{central fiber} (The cohomology ring of $N_{dR}$).} 
 Let $p > 0$ and assume the same conditions on $r$ and $d$ seen above: $d=\overline{d}p$ and $(r,d)=1$.
  We  use (\ref{to ex}) from the proof of Theorem \ref{zmain}, to 
  prove  that there is a canonical  filtered isomorphism of cohomology rings
\beq\la{mdrndr}
H^* \left(M_{dR}(C;d), \oql\right) \simeq H^*\left(N_{dR}(C;d), \oql\right), 
\eeq
where $N_{dR}$ is the   subspace of stable connection with nilpotent $p$-curvature, i.e.
 the fiber over the origin of the de Rham-Hitchin morphism 
$h_{dR}: M_{dR} \to A(\FC)$ (\S\ref{rino}). The corresponding fact for $M_{Dol}$ and the fiber $N_{Dol}$ is well-known and valid without any assumptions on rank and degree, and it can be proved by using the theory of weights  jointly with the classical   contracting $\Gm$-action on the $\Gm$-equivariant and proper Hitchin morphism  $h_{Dol}: M_{Dol} \to A(C)$. The surprising aspect of (\ref{mdrndr}) is that there is no known $\Gm$-action on $M_{dR}$.

{\bf Theorem \ref{int cons} ($p$-Multiplicative periodicity with Frobenius twists).}  
Let $p > 0$ and assume the same conditions on $r$ and $d$ seen above: $d=\overline{d}p$ and $(r,d)=1$. This theorem expresses 
a new periodicity feature concerning the cohomology rings of Dolbeault moduli spaces for degrees
that differ by a multiple a  power of the characteristic $p>0,$ namely,  there is a canonical  filtered isomorphism of cohomology rings 
\beq\la{p-per}
H^*\left(M_{Dol}(C; \overline{d}), \oql\right) \simeq H^*\left(M_{Dol}(C^{(-m)}; \overline{d}p^m), \oql\right), 
\eeq
where $m\geq 0$, and $C^{(-m)}$ is  the $(-m)$-th Frobenius twist of $C$, i.e. the base change of $C/\field$ via the $m$-th power $fr_{\field}^{-m}: \field
\stackrel{\sim}\to \field$, $a \mapsto a^{p^{-m}}$, of the inverse of the absolute Frobenius automorphism $fr_\field$.

{\bf Theorem \ref{vvooqq} (Different curves, same degree).}  
Let $p\geq 0$ and let  $(r,d)=1$. We do not assume that the degree is a multiple of $p$. We prove that the cohomology rings of the Dolbeault moduli spaces of two curves $C_i$ of the same genus  are non canonically  filtered-isomorphic
\beq\la{fr4}
H^*(M_{Dol} (C_1 ; d)) \simeq H^*(M_{Dol} (C_2 ; d)).
\eeq
 Over the complex numbers:
the statement without the filtrations is an easy consequence of the fact that the two Dolbeault moduli spaces are diffeomorphic to the (common) Betti moduli space; the filtered statement is proved in \ci{dCM}.

{\bf Theorem \ref{e4} ($p$-Multiplicative periodicity without Frobenius twists).} 
Let $p > 0$ and assume the same conditions on $r$ and $d$ seen above: $d=\overline{d}p$ and $(r,d)=1$. 
We prove  a non canonical analogue of (\ref{p-per}), with the Frobenius twist $C^{(-m)}$ replaced by the original curve $C$ (or, in fact,  by any curve
of the same genus, in view of Theorem \ref{vvooqq})
\beq\la{p-per-2}
H^*\left(M_{Dol}(C; \overline{d}), \oql\right) \simeq H^*\left(M_{Dol}(C; \overline{d}p^m), \oql\right).
\eeq

{\bf Theorem \ref{c cop} (Same curve, different degrees; $char (\field) =0$).} 
Here, $p=0$. Let $d,d'$ be degrees coprime to the rank $r.$
We prove that the cohomology rings of the Dolbeault moduli spaces in degrees $d,d'$
for a curve $C$  are filtered  isomorphic
\beq\la{to7}
H^*(M_{Dol}(C,d), \oql) \simeq H^*(M_{Dol}(C,d'), \oql).
\eeq
Over the complex numbers,
the statement without the filtrations is a consequence of the fact that the two Dolbeault moduli spaces are diffeomorphic to their Betti  counterparts and that, in turn, these are Galois-conjugate.
The resulting ``transcendental" isomorphism differs from the isomorphism  in Theorem \ref{c cop}.
Presently,  it is not known how to compare the perverse Leray filtrations
under the ``transcendental" isomorphism.  

Added in revision. 1) This comparison is the subject of \ci{dmsz}: the two match.
2) In the recent paper by T. Kinjo and N. Koseki \ci[Thm. 1.1]{kk}, an isomorphism of the form (\ref{to7}) is obtained by a method that differs  from ours.

{\bf Theorem \ref{c cop p} (Same curve, different degrees; $char (\field) =p>0$).} 
Here, $p>0.$ Let $d,d'$ be degrees coprime to the rank $r$
and assume $p>r$. Then we prove the statement analogous to Theorem \ref{c cop}. 

We want to emphasize the  following amusing fact:  Theorem \ref{e4} (a result in positive characteristic) is used
to prove Theorem \ref{c cop} (a result in characteristic zero);  in turn, this latter result is used to prove
Theorem \ref{c cop p} (a result in positive characteristic).

{\bf Acknowledgments.} 
We thank the referee for the excellent suggestions.
We are very grateful to Mircea~Musta\c{t}\u{a} for providing us with a proof of the properness criterion
afforded by Proposition \ref{ad hoc pr}.
We are also very grateful to Michael Groechenig for many inspiring conversations on the subject.
We thank Dan Abramovich,   Barghav Bhatt, H\'el\`ene Esnault, Jochen Heinloth, 
Luc Illusie, Adrian Langer, Davesh Maulik, Junliang Shen,  Ravi Vakil
and Angelo Vistoli for very useful and pleasant  email and Zoom exchanges.
 M.A. de Cataldo is  partially supported by NSF grant DMS  1901975 and by a Simons Fellowship in Mathematics.
 S. Zhang is partially supported by NSF grant DMS  1901975.
M.A. de Cataldo dedicates this paper to the memory of his parents, with love.

\subsection{Notation and preliminaries}\la{bzzz}$\;$

{\bf The schemes we work with.}
We fix a base ring $\brg$ that is either a field, or a 
discrete valuation ring (DVR), possibly of mixed characteristic $(0,p>0)$. We work with separated schemes of finite type
over $\brg$,  and with   $\brg$-morphisms that are separated and of finite type.
The term variety is reserved to schemes as above when the base is a field.  

{\bf Constructible derived categories and perverse $t$-structures over the DVR.} 
Let $\ell$ be a prime number invertible in $J$.
We employ  the usual formalism of the corresponding 
``derived" categories $D^b_{c}(-,\oql)$ of bounded constructible ``complexes" of $\oql$-adic sheaves
endowed with the appropriate version of the middle perversity $t$-structure: the classical one if $J$ is a field;
the rectified one  if $J$ is a DVR as above.
When working  over a field with the usal six functors and the perverse $t$-structure, the references 
  \ci[Thm. 6.3]{ek} and \ci{bbdg} are suffcient for our purposes. When working over a DVR as above, we need complement these references  so that we can work with nearby/vanishing cycles functors and their $t$-exactness properties for the rectified perverse $t$-structure. For a discussion and additional references, see \ci[\S5.2]{de-zh B}.

{\bf The perverse Leray filtration.}
\'Etale cohomology groups are taken only for varieties over algebraically closed fields $J=\field.$ More often than not, we  drop ``\'etale."
 Let $f: X \to Y$ be a  $\field$-morphism  and let $K\in D^b_{c}(X,\oql).$ We denote the functor $Rf_*$ simply by $f_*$; the derived direct images are denoted by $R^\bullet f_*$, for $\bullet \in \zed.$ We denote the perverse truncation functors $\ptd{\bullet}$, for $\bullet \in \zed.$
 The increasing perverse Leray
filtration $P^f_\bullet$ on $H^\star (X,K)$ is defined by setting, for every $\bullet,\star \in \zed$ 
\beq\la{def plf}
P^f_\bullet H^\star (X,K):= \im \{H^\star (Y, \ptd{\bullet} Rf_* K) \to H^\star (Y, Rf_* K)=H^\star (X,K)\}.
\eeq
Let $f:X\to Y$ and $g:Y  \to Z$ be morphisms of $\field$-varieties. If $g$ is finite,  then $g_*$ is $t$-exact
(hence, being cohomological, exact on the category of perverse sheaves),
so that
\beq\la{comp pt}
\xymatrix{
P^{g\circ f}_\bullet H^\star (X,K) \,=\, P^{f}_\bullet H^\star (X,K). 
}
\eeq

{\bf \'Etale cohomology rings.}
When working with separated schemes of finite type (varieties) over an algebraically closed field $\field$ of positive characteristic $p>0$,
we fix any other prime   $\ell \neq p$.
The graded  \'etale cohomology groups $H^* (-, \oql)$ of such a variety  form a  unital, associative, graded-commutative $\oql$-algebra for the cup product operation.  A graded morphism  between the graded \'etale cohomology groups of two varieties   preserving these structures is simply called a morphism of  cohomology rings. Of course, pull-backs via morphisms are examples. In this paper,
we find isomorphisms of cohomology rings, with additional compatibilities, that do not arise from  morphisms.

\subsection{Reminder on vanishing/nearby cycles, and specialization in cohomology}\la{rem vn}$\;$

We briefly recall the general set-up for the formalism of nearby-vanishing cycles using strictly Henselian traits; 
see  \ci{sga72,illusie} and \ci[p.214, Remark]{ek}. Caveat: there are several  distinct and all well-established ways to denote nearby/vanishing cycles in the literature;  our notation $\phi$ for the vanishing cycle differs by a shift (our $\phi[1]$ is their $\phi$) with respect to the given references; our current notation makes $\phi$ and $\psi[-1]$ $t$-exact functors, and is in accordance with \ci{de-2021, de-zh B}, as well as with  other occurrences in the literature.

{\bf Strictly Henselian traits.}
Let $(S,s,\eta ,\ov{\eta})$ be a strictly Henselian trait together with  a minimal choice of generic  geometric point, i.e.:
\ben
\item $S$ is the spectrum
of a   strictly Henselian    discrete valuation ring, hence   with separably closed residue field;
\item
 $i: s\to  S$ is the closed point (it is also a geometric point);
 \item
 \begin{equation}
 \label{jbar}
 \bar{j}: \ov{\eta} \to \eta  \to S
 \end{equation} 
 is the generic point of $S$,  with the associated geometric point stemming from a fixed choice of a   separable closure  $k(\eta)^{\rm sep}/k(\eta)$ of the fraction field of the Henselian ring.
 \een
 The objects restricted via the base change  $i:s \to S$  are denoted by a subscript $-_s$, and similarly for $-_\eta$ and for $-_{\ov{\eta}}$.

{\bf Vanishing/nearby cycles.}
Let $v:X \to S$ be a morphism of finite type. We have the distinguished triangle of functors
\[
\xymatrix{
i^* \ar[r]   & \psi_v \ar[r] & \phi_v [1]  \ar@{~>}[r] &,
}
\]
where the three functors  are functors $D^b_{c}(X,\oql) \to D^b_{c} (X_s, \oql)$. The functor $\psi_v$ is called the nearby cycle functor
and the functor $\phi_v$ is called the vanishing cycle functor.
By restricting to $\eta$, we can also view the functor $\psi_v$ as a functor $D^b_{c} (X_\eta) \to D^b_{c}(X_s)$.
If $\eta^*F \simeq \eta^*G,$ then $\psi_v(F) \simeq \psi_v (G)$, functorially.

{\bf The specialization morphism ${\rm sp}$.}
For $F$ in  $D^b_{c}(X)$, we have the fundamental diagram
\beq\la{rr12}
\xymatrix{
H^* (X_s, F) & H^* (X,F) \ar[l]_-{i^*} \ar[r]^{\ov{\eta}^*} & H(X_{\ov{\eta}}, F).
}
\eeq
If  $i^*$ is an isomorphism, then we define the specialization morphism by setting
\beq\la{spz1}
\xymatrix{
R^\bullet  v_* i^* F = H^\bullet  (X_s, F) \ar[rrr]^-{{\rm sp}:= \,\ov{\eta}^* \circ {(i^*)^{-1}}} && &  H^\bullet (X_{\ov{\eta}}, F) = \m{H}^\bullet_{s} (\psi_v v_* F), & \forall \bullet \in \zed.
}
\eeq
 By the Proper Base Change Thoerem, if $v$ is proper, then $i^*$ is an isomorphism and the specialization morphism is defined. However, it  $v$ is not proper, then $i^*$ may fail to be an isomorphism and  the specialization morphism may fail to be defined. \ci{de-2021} is devoted to explore this phenomenon, and in this paper, we work in such a situation.

\begin{rmk}\la{comp cup}
If the specialization morphism is defined, then it is compatible with cup products, e.g. when $F=\oql$.
More generally, it is compatible with pairings $F'\otimes F'' \to F$ of objects in $D^b_{c}(X)$ \ci[\S4.3]{illusie}.
\end{rmk}

\begin{fact}\la{fvz}
For the purpose of  this paper, the most important properties  of the vanishing cycle functors are:

\ben
\item
 If $v$ is smooth, then $\phi_v (\oql)=0$;
see \ci[XIII, Reformulation 2.1.5]{sga72}.

\item
If $f: Y\to X$ is a proper morphism, and  $u:Y \to S$ and $f_s: Y_s \to X_s$ 
are the resulting morphisms,  then, by proper base change, we have  natural isomorphisms $\phi_v   f_* = f_{s,*} \phi_u$ and $\psi_v   f_* = f_{s,*} \psi_u$ (\ci[XIII, (2.1.7.1)]{sga72}).
\een
\end{fact}

The moduli spaces we work with are not proper over their base, so that it is not clear at the outset that the various specialization morphisms
we wish to consider are even defined. In this context, we prove Proposition \ref{sp cptz} for use in \S\ref{omq2}.
On the other hand,  in \S\ref{the mainz}, we circumvent the direct use of these specialization morphisms; see the proof of Theorem \ref{zmain}.

\subsection{The moduli spaces we work with}\la{rino}$\;$

 The existence, quasi projectivity, and uniform  (universal  in the coprime case when not in characteristic zero) corepresentability of the moduli  spaces we are about to introduce
 have been established by C. Simpson \ci{si mo I, si mo II}  for smooth projective families over a base  of finite type over a ground  field of characteristic zero, and over a base of finite type over a universally Japanese ring by A. Langer \ci[Theorem. 1.1]{la-2014}.  Recall that ``universal" (``uniform," resp.) refers to the commutation of the formation
 of the coarse moduli space with arbitrary (flat, resp.) base change.

 {\bf Base over base ring.}
 In this paper, we only need to consider the set-up of a base $\base$ that is Noetherian, and of finite type
 over  a base ring $\brg$, that is either an algebraically closed field $\field$, or a DVR.
 For a  more general setup and more details concerning the moduli spaces we use, see \ci{de-zh B}.
  Note that for the sake of the existence of the moduli spaces, the assumption on the base has been relaxed to $B$ being any noetherian scheme in Langer's recent paper \ci[Theorem 1.1]{la-2021}.

{\bf Smooth curves.}
In this paper, a smooth curve $\csub$ is a smooth projective  morphism $C\to B$  with geometric fibers  integral of dimension one.
If the base $\base = \brg =\field$ is a field, then we often write $C$ instead of $C/\field.$

{\bf Coprimality assumption on rank, degree, and characteristic of the ground field.} 
When working with vector bundles, we denote their rank by $r$, and their degree by $d$. In this paper, we always assume they are coprime, i.e. ${\rm g.c.d.}(r,d)=1.$
When working with  the de Rham moduli space of stable (=semistable)   connections on a smooth
curve over an algebraically closed field of positive characteristic $p>0$,  we always assume,
in addition,  that the degree $d=\ov{d}p$ is an integer  multiple of the characteristic $p$; otherwise, there are no such  connections.  Our assumptions imply   that stability coincides with semistability thus ensuring:
 1)  the nonsingularity of the  Hodge ($t$-connections), Dolbeault (Higgs bundles) 
and de Rham (connections) moduli spaces (cf. \S\ref{smothh mdr}); 2)   that these moduli universally (instead of merely uniformly)
corepresent their moduli functor (\ci[Tm. 1.1]{la-2014}), so that
the formation of such moduli spaces commutes with arbitrary base change into  the moduli space, hence in particular  into $\base$, or $\brg$.

Regrettably, the coprimality assumptions
 rules out the  important case of connections of degree zero. On the other hand, these assumptions are the most natural  when dealing
 with nonsingular moduli spaces. While our methods require 1) and  2) above, one wonders if many of the result of this paper hold
 without the coprimality assumption, i.e. for the possibly singular  Hodge/Dolbeault/de Rham moduli spaces that arise. 
 We are not sure what to expect in the singular case. Note also
 that the ``p-multiplicative periodicity" results Theorems \ref{int cons} and \ref{e4}  express a property of the Dolbeault moduli spaces
  that acquires a non trivial meaning only in non zero degrees; similarly, for Theorems \ref{c cop} and \ref{c cop p}. 
 
 
{\bf The Hodge moduli space.}
A $t$-connection on a smooth curve  $\csub$ is a triple $(t, E, \nabla_t)$, where $t$ is a regular function on $\base$,
$E$ is a vector bundle on $C$, $\nabla_t: E \to E \otimes_{\m{O}_C} \Omega^1_{\csub}$ is $\m{O}_\base$-linear and satisfies
the twisted Leibnitz rule  $\nabla_t (f \s)=  tdf \otimes \s + f \nabla_t(\s)$, for every local function $f$ on $C$, and every local section $\s$ of
$E$ on $C$ on the same open subset.
There is the  quasi-projective $\base$-scheme $M_{Hod}(\csub ;r,d)$ (cf. \ci[Thm. 1.1]{la-2014}), coarse Hodge moduli space universally corepresenting   slope stable  $t$-connections of rank $r$ and degree $d$  on the  smooth curve $\csub$. It comes with a natural $\base$-morphism of finite type to the affine line assigning $t$  to a $t$-connection  
\beq\la{mz1}
\xymatrix{
\tau_{Hod} (\csub;r,d): M_{Hod}(\csub;r,d) \ar[r] &
\Aa^1_\base.
}
\eeq

{\bf Dolbeault moduli space and Hitchin morphism.} By the universal corepresentability property,
if we take the fiber over the origin  $0_\base \to \Aa^1_\base$, then we obtain the   quasi-projective $B$-scheme 
\beq\la{mdol}
M_{Dol}(\csub ;r,d),
\eeq
coarse Dolbeault moduli space universally corepresenting  slope stable 
rank $r$ and degree $d$  Higgs bundles, twisted by the canonical bundle,  on the family of curves $\csub$. 
If $\base$ is a field, then the Dolbeault moduli space is empty if an only if the genus of the curve is zero  and the rank $r \geq2;$ otherwise,
this moduli space  is integral, nonsingular, and  of dimension that depends only on the rank $r$  and genus $g$  of the curve
(cf. \ci[\S7]{nitsure})
\beq\la{la dim e} 
\dim M_{Dol}(C,r,d)=  r^2 (2g-2)+2.
 \eeq 
 Let $A(\csub;r)$ be the vector bundle on $B$  of rank one half  the dimension (\ref{la dim e}), with fiber 
 $H^0(C_b, \oplus_{i=1}^r \omega_{C_b}^{\otimes i}).$
 There is the  projective and surjective   Hitchin $\base$-morphism   
 \beq\la{mhi}
\xymatrix{
h_{Dol} (\csub;r,d): M_{Dol}(\csub ;r,d) \ar[r] & A(\csub; r),
}
\eeq
assigning to a Higgs bundle, the characteristic polynomial of its Higgs field. 
For the projectivity of the Hitchin morphism over a base, see \ci[Th. 2.18]{de-zh B}.

{\bf The Hitchin base.}
The $\base$-scheme  $A(\csub;r)$ is sometimes called the Hitchin base,
or the space of characteristic polynomials of rank $r$  Higgs fields, or the space of degree $r$ spectral curves  over $\csub$.

{\bf de Rham moduli space and de Rham-Hitchin morphism.} 
If we take the fiber  of (\ref{mz1}) over   $1_{\base} \to \Aa^1_{\base}$, then we obtain the   quasi-projective $\base$-scheme 
\beq\la{mdr}
M_{dR}(\csub;r,d), 
\eeq
coarse de Rham moduli space, universally corepresenting slope  rank $r$ and degree $d$ 
stable  connections   on the family of curves $\csub$.  

If $\brg= \field$ is  an algebraically closed   field of characteristic zero, 
then  the de Rham moduli space  is non-empty iff and only if $d=0$. 

If $\brg =\field$ is   an algebraically closed  field of positive characteristic $p$, then  
the de Rham moduli space is non empty if and only if $d=\ov{d}p$ is an integer multiple of $p$ (recall that this is part of our assumptions on rank, degree and characteristic); see \ci[Pr. 3.1]{bi-su}.  In this case,  it is shown in Lemma \ref{nonempty} that the de Rham moduli space  is integral, nonsingular,
of the same dimension (\ref{la dim e}) as the Dolbeault moduli space for  the same rank and degree. 
In this case we also have the
projective and surjective  de Rham-Hitchin  $\base$-morphism
\beq\la{mhidr}
\xymatrix{
h_{dR} (\csub;r,\ov{d}p): M_{dR}(\csub ;r,\ov{d}p) \ar[r] & A(\csubf; r),
}
\eeq
where ${\csubf}$ is the base change of $\csub$ via the absolute Frobenius endomorphism $fr_{\base} : B \to B$
(absolute Frobenius for $B$: identity of topological space; functions raised to the $p$-th power). 
The de Rham-Hitchin morphism is defined in \ci[Def. 3.16]{gr-2016}. It is shown to be proper in \ci[Cor. 3.47]{gr-2016}, thus projective in view of the quasi-projectivity  at the source.
For every closed point $b \in B$, we have that the fiber
$(\csubf)_b  = (C_b)^{(1)} =: \kappa(b) \times_{\kappa (b), fr_{\kappa (b)}} C$ is the Frobenius twist of the curve  $C/\kappa (b)$, i.e. the base change of $C/\kappa (b)$ via the
absolute Frobenius automorphism $fr_{\kappa (b)}$ of $\kappa (b)$.
The   fiber   at $b\in B$ of the vector bundle 
$A(\csubf;r)$ is  given by
 $\oplus_{i=1}^r H^0(C_{b}^{(1)}, \omega^{\otimes i}_{C^{(1)}_b})$.
 
{\bf Hodge-Hitchin morphism ($char (\field)=p>0$).} Let $\brg =\field$ be an algebraically closed field of  positive characteristic  $p>0.$
Y. Lazslo and C. Pauly \ci{la-pa} (see also \ci{de-zh B}) have constructed a natural  factorization of the morphism $\tau_{Hod}$ (\ref{mz1}) 
\beq\la{mz2}
\xymatrix{
\tau_{Hod} (\csub;r,d): M_{Hod}(\csub;r,d) \ar[rr]^-{h_{Hod} (\csub;r,d)} &&
A(X^{(B)}/B;r)  \times_{\base} \Aa^1_{\base} \ar[r]^-{\rm pr_2} &
\Aa^1_{\base}.
}
\eeq
We call the  quasi-projective $B$-morphism 
$h_{Hod}(\csub;r,d)$ the Hodge-Hitchin morphism. It assigns to a $t$-connection on a curve $C$, the characteristic polynomial of its $p$-curvature:
the $p$-curvature is an Higgs field on the same underlying vector bundle on the curve $C$, but for the $p$-th power of the canonical line bundle; the key observation is that this characteristic polynomial  is the  pull-back via the relative Frobenius
morphism  $Fr_C: C\to \FC$ of a uniquely determined  characteristic polynomial  on $\FC$.

If we specialize $h_{Hod}(\csub;r,d)$  at $1_{\base}$, then we obtain the de Rham-Hitchin morphism
\beq\la{d11}
\xymatrix{
h_{dR}(\csub; r,d): = h_{Hod}(\csub;r,d)_{1_{\base}}: M_{dR}(\csub;r,d) \ar[r] & A(X^{(B)}/B;r).
}
\eeq
If we specialize $h_{Hod}(\csub;r,d)$  at $0_{\base}$, then we obtain the classical  Hitchin morphism post-composed with the 
Frobenius relative to $B$ (see  \ci{de-zh B})
\beq\la{d100m}
\xymatrix{
h_{Hod}(\csub;r,d)_{0_{\base}}:
M_{Dol}(\csub;r,d) \ar[rr]^-{h_{Dol}(\csub;r,d)} && A(\csub;r) \ar[rr]^-{Fr_{A(\csub;r)/B}}&& A(X^{(B)}/B;r).
}
\eeq

{\bf $\Gm$-actions and equivariance.}
The group scheme ${\Gm}_{,\base}$ acts on the Hodge moduli space by weigth $1$ dilatation on the $t$-connections: $\lambda \cdot \nabla_t:=
\nabla_{\lambda t}$, and similarly on $\Aa^1_{\base}.$ The morphism $\tau$ (\ref{mz1}) is ${\Gm}_{,\base}$-equivariant for these actions.
Moreover, the pre-image of ${\Gm}_{,\base} \subseteq \Aa^1_{\base}$ is canonically and ${\Gm}_{,\base}$-equivariantly
a fiber product over $B$  of the de Rham moduli space times ${\Gm}_{,\base}$, i.e. we have 
(see  \ci{de-zh B})
\beq\la{prfr}
\tau^{-1}({\Gm}_{,\base})
\simeq M_{dR}(\csub)\times_{\base} {\Gm}_{,\base}.
\eeq

If $\brg= \field$ is an algebraically closed field of positive characteristic $p>0$, then the group scheme ${\Gm}_{,\base}$ acts on $A(\csubf;r)  \times_{\base} \Aa^1_{\base}$ as follows: by weigth $1$ dilations
on $\Aa^1_{\base}$; by weight $ip$ dilations on each term $H^0(C_{b}^{(1)}, \omega^{\otimes i}_{C^{(1)}_b}).$

If $\brg$ is arbitrary, then the group scheme ${\Gm}_{,\base}$ acts on $A(\csub;r)  \times_{\base} \Aa^1_{\base}$ in a similar way, but by with weight
$i$ diltations on each term  $H^0(C_{b}, \omega^{\otimes i}_{C_b}).$

All the  morphisms appearing  in (\ref{mz2}), (\ref{d11}) and  (\ref{d100m}) are ${\Gm}_{,\base}$-equivariant for specified actions.
Moreover, the trivialiazation  (\ref{prfr}) extends to an evident ${\Gm}_{,\base}$-equivariant  trivialization of (\ref{mz2})
over ${\Gm}_{,\base} \subseteq \Aa^1_{\base}$ and, in particular, we have a natural ${\Gm}_{,\base}$-equivariant identification
\beq\la{hh91}
h_{Hod}|{\Gm}_{,\base} = h_{dR} \times_{\base}  {\rm Id}_{{\Gm}_{,\base}}.
\eeq

Even without the coprimality assumption, the following properness statement is proved
in \ci[Thm. 2.13.(2)]{de-zh B}, and it can also be seen as a consequence of what is stated  in   \ci[top of p. 321]{la-2014}. We thank A. Langer for providing us with a proof in a private communication (Added in revision: A. Langer's communication now appears  in \ci[Thm. 1.3]{la-2021}).
This properness result  plays an essential role in this paper.
An alternative proof of this properness under our coprimality assumptions is given  in Proposition \ref{hh is proper} which, in turn,  is based on the ad hoc criterion Proposition \ref{ad hoc pr}.

\begin{tm}\la{hohiisproper}
The Hodge-Hitchin morphism $h_{Hod}$ (\ref{mz2}) is proper, in fact projective.
\end{tm}

\subsection{Smoothness of  moduli spaces}\la{smothh mdr}$\;$

In this section, we place ourselves in  the following special case of the set-up in \S\ref{rino}:
$C= C/\field$ is  a smooth curve over an algebraically closed field $\field$ of positive characteristic $p$,
the degree $d=\ov{d} p$ is an integer multiple of the characteristic and ${\rm g.c.d.} (r,d)=1.$ 

The aim is to prove Proposition \ref{t is smooth}, to the effect that under these coprimality conditions the morphism $\tau_{Hod}(C;r, \ov{d}p)$ (\ref{mz1}) 
is smooth. This smoothness is essential to the  approach we take in this paper via vanishing/nearby cycle functors.

\begin{lm}[{\bf (Smoothness of $M_{dR}$)}]\la{nonempty}
The moduli space $M_{dR}(C;r,\ov{d}p)$ of stable  connections   is non empty, integral,   quasi-projective, non-singular, of the same dimension  (\ref{la dim e}) of the corresponding moduli space
$M_{Dol}(C;r, \ov{d}p)$ of  stable Higgs bundles of the same degree and rank.
In particular, the fibers of the morphism $\tau_{Hod} (C;r, \ov{d}p)$   (\ref{mz1}) over the geometric points 
of $\Aa^1_\field$ are integral, nonsingular of the same dimension (\ref{la dim e}).
\end{lm}
\begin{proof} We drop some decorations.
The fiber of $\tau$ over the closed point $0$ is $M_{Dol}$, and the fibers  over the other closed points
are isomorphic to $M_{dR}$ in view of the trivialization  (\ref{mz2}). We are thus left with  proving the assertions for the fiber $M_{dR}$.

Let
 $C^{(1)}$ be the Frobenius twist of the curve $C$. 
Note that $r$ and $\overline{d}:=d/p$ are also coprime.
As recalled in \S\ref{rino},  the moduli space $M_{Dol}(C^{(1)}; r, \ov{d})$ is non-empty, integral,  quasi-projective  nonsingular
of dimension (\ref{la dim e}).
Since its dimension depends only on the genus  $g(C)=g(C^{(1)})$ of the curve $C$,  and on the rank  $r$ (cf. \ci[Prop. 7.4]{nitsure}), we have that 
$M_{Dol}(C;r, \ov{d}p)$ and $M_{Dol}(C^{(1)}; r, \ov{d})$ have the same dimension  (\ref{la dim e}).

Let $h_{Dol}(C^{(1)}, r, \ov{d}): M_{Dol}(C^{(1)}, r, \ov{d}) \to A(C^{(1)}, \omega_{C^(1)},r)$ be the Hitchin morphism for stable
Higgs bundles for  the canonical  line bundle  on  $C^{(1)}$. Since stability and semistability coincide by coprimality, 
this Hitchin morphism is proper (\ci[Th. 6.1]{nitsure}), and  in fact projective,
since the domain is  quasi projective.
Since the general fiber is connected, being the Jacobian of a nonsingular spectral curve (\ci[Prop. 3.6]{be-na-ra}), and the target is
nonsingular, hence normal, this Hitchin morphism has connected fibers \cite[\href{https://stacks.math.columbia.edu/tag/03H0}{03H0}]{stacks-project}. Being proper and dominant, it is also surjective.

Let $h_{dR}(C;r,d): M_{dR}(C; r, d) \to A(C^{(1)}; r)$ be the de Rham-Hitchin morphism for stable connections
on the curve $C$. This morphism is defined in \ci[Def. 3.16, p.1007]{gr-2016}.  As seen in \S\ref{rino}, it coincides with the specialization
at $t=1$ of the Hodge-Hitchin morphism $h_{Hod}(C;r,d)$.

By combining \ci[Th. 1.1, Cor. 3.45 and Lm. 3.46]{gr-2016}, the two morphisms $h_{Dol}(C^{(1)}; r, \ov{d})$ and $h_{dR}(C, r,\ov{d}p)$
are \'etale locally equivalent over the base $A(C^{(1)}; r)$. 

As noted in \ci[Cor. 3.47]{gr-2016},  this \'etale  local equivalence implies that
the de Rham-Hitchin morphism is proper and surjective. In fact, the de Rham-Hitchin morphism  is projective in view of the quasi-projectivity of domain and target.

This \'etale local equivalence also implies that $M_{dR}(C; r, \ov{d}p) $ is nonsingular of pure  dimension $\dim M_{Dol}(C^{(1)}; r,\ov{d})
=\dim M_{Dol}(C;r,d)$ (\ref{la dim e}). By coupling the \'etale local equivalence with
the connectedness of the fibers, and with the integrality of $M_{Dol}(C^{(1)}; r, \ov{d} )$, we deduce that $M_{dR}(C, \ov{d}p)$ is integral as well.
\end{proof}

\begin{pr}[{\bf (Smoothness of $\tau_{Hod}: M_{Hod} \to \Aa^1_\field$)}]\la{t is smooth}
The morphism $\tau_{Hod} (C; r, \ov{d}p)$ (\ref{mz1})  is  a smooth fibration, i.e. smooth, surjective, with connected  fibers,
onto the affine line $\Aa^1_\field$.
The Hodge moduli space  $M_{Hod}(C; r, \ov{d}p)$ of stable  pairs is integral and nonsingular.
\end{pr}
\begin{proof} We drop some decorations. In particular, let us simply write $\tau:M\to \Aa^1_\field.$ 
Since the fibers of $\tau$ are smooth (Lemma \ref{nonempty}), in order to prove that $\tau$ is smooth, it is enough to prove that $\tau$ is flat.
Once $\tau$ is smooth, the smoothness and integrality  of $M$ follow from the flatness of $\tau$ and the smoothness and integrality 
of the target and  of the  fibers of $\tau.$ 

We know that the fibers of  $\tau$ are nonsingular, integral and  of  dimension (\ref{la dim e}) (Lemma \ref{nonempty} and (\ref{prfr})).
However, off the bat, we are unaware of an evident reason why $M$ should be irreducible, or even reduced. 

We know that $\tau$ is flat over ${\Gm}_{,\field} \subseteq \Aa^1_\field$ by virtue of the trivialization (\ref{prfr}).
We need to verify that $\tau$ is flat over the origin. 
This is a local question near the origin $0 \in \Aa^1_\field$.

Let $A:={\rm Spec} (\field [x]_{(x)})$  (Hitchin bases, typically also denoted by $A$ in this paper, do not appear in this proof)  be the spectrum of the local ring of $0 \in \Aa^1_\field$ and let $\tau_A: M_A \to A$ be the base change of 
$\tau$ via $A \to \Aa^1_\field.$ We need to show that $M_A/A$ is flat. 

The scheme $M_A$ universally corepresents suitable equivalence classes of semistable
$t$-connections  on $A\times C.$

Note that $\tau_A$ is surjective, hence dominant. 
 Let $0$ and $\alpha$ be the closed and open points in $A$, respectively. Let $(M_A)_0=M_0$ and $(M_A)_\alpha=M_\alpha$
 be the corresponding fibers.

{\bf CLAIM 1}: {\em  We have $\ov{M_\alpha} \cap (M_A)_0 \neq \emptyset.$}
Let $E$ be a rank $r$ and degree $\ov{d}p$ stable vector bundle on $C$ (there are such bundles  since  their moduli space is
an irreducible nonsingular variety of positive dimension one half of (\ref{la dim e})).  The stable bundle $E$ is indecomposable \ci[Cor 1.2.8]{hu-le}.
By \ci[Prop.3.1]{bi-su} the vector bundle $E$ admits flat connections $\nabla$. Let $(\m{E}, x\nabla)$ be the $t$-connection
on $A\times C$ obtained by  pulling back $(E,\nabla)$ via the projection onto $C$  and by twisting the connection
by the function $x$ on $A.$ By \ci[Prop. 1.3.7]{hu-le}, we have that $\m{E}$, being stable on the geometric fibers,   is a stable bundle 
on $A\times C$, so that $(\m{E}, x\nabla)$ is a stable $t$-connection on $A\times C.$ We thus have that $(\m{E}, x\nabla) \in M(A)$.
Then $(E, (x=0)\nabla)=(E,0) \in M(\field)$ is a specialization of the restriction of $(\m{E}, x\nabla)$ to the generic point of $A.$
This proves CLAIM 1.

{\bf CLAIM 2:} {\em We have $(\ov{M_\alpha})_0 = (M_A)_0.$}
The closure $\ov{ M_\alpha } $ is integral and it is a closed subscheme of $M_A.$
It follows that the  first fiber is a closed (and non-empty by CLAIM 1) subscheme of the integral nonsingular  second fiber.
By the upper-semicontinuity of the dimension of fibers at the source, the two fibers have the same dimension, hence they coincide
by the integrality of the second fiber. This proves CLAIM 2. 

{\bf CLAIM 3:} {\em We have the equality of integral schemes  $\ov{M_\alpha} = M_{A,red}$.}
The first is a  closed  and dense (CLAIM 2 implies they have the same geometric points) subscheme of the second, which is also integral.
CLAIM 3 is proved.

By \ci[III.9.7]{ha}, we have that $M_{A,red} \to A$, and thus $M_{red}\to \Aa^1_\field$, are flat.

It remains to show that $M_A$ is indeed reduced: Let $U$ be any nonempty affine open subset of $M_A$. Assume $f\in \Gamma(U,\mathcal{O}_{U})$ is a nonzero nilpotent element so that $f$ maps to $0\in \Gamma(U_{red},\mathcal{O}_{U_{red}})$. We have the factorization $f=x^N g$ where $g\notin (x)\cdot\Gamma(U,\mathcal{O}_{U})$. By CLAIM 3, we have that $U_{red}$ is integral. Therefore either $x$ or $g$ is nilpotent in $\Gamma(U,\mathcal{O}_{U})$. Since $M_A\to A$ is dominant, we have that $x$ is not nilpotent in $\Gamma(U,\mathcal{O}_{U})$. Thus $g$ is nilpotent. Since $g\notin (x)\cdot \Gamma(U,\mathcal{O}_{U})$, $g$ maps to a nonzero nilpotent element in the special fiber of $M_A$ over $A$, which contradicts the integrality of $(M_A)_0$.
\end{proof}

\subsection{Ad hoc proof of the properness of the Hodge-Hitchin morphism}\la{ad hocprp}$\;$

The purpose of this section is to give a proof   (Proposition \ref{hh is proper})  of the properness of the Hodge-Hitchin morphism (Theorem
\ref{hohiisproper})
in the cases we need in this paper. The proof is based on the application of the  following  rather general properness criterion, and is based on the knowledge that  the Hitchin and the de Rham-Hitchin morphisms are proper.
In some sense, we collate these two properness statements. On the other hand, this collation does not seem to be immediate; see Remark \ref{need hyp}.
We are very grateful to Mircea~Musta\c{t}\u{a} for providing us with a proof  of said criterion.
We are also very grateful to Ravi Vakil for pointing out some counterexamples to some overly optimistic earlier versions of this criterion.

\begin{pr}[{\bf (An ad hoc properness criterion)}]\la{ad hoc pr}
Let $ m\circ f: X\to Y \to T$ be  morphisms of schemes. We assume that
\ben
\item 
$X$ is quasi-compact and quasi-separated, and $Y$ is noetherian;
\item
$X$ and $Y$ are  integral, and $Y$ is normal;
\item
$f:X \to Y$ is separated, of finite type,  surjective and with  geometrically connected fibers;

\item
for every closed point $t \in T,$ the morphism $f_t: X_t \to Y_t$ obtained by base change is proper.
\een
Then $f$ is proper.
\end{pr}
\begin{proof}
Let $y \to Y$ be a closed point. The fiber $f^{-1}(y)=X_y \to y$ is proper, as it is the fiber over $y$ of the morphism $f_t: X_t \to Y_t,$ with $t:= m(y).$
It follows that it is enough to prove the Proposition when $m:Y\to T$ is the identity morphism.   
We assume we are in that case.

We have the following commutative diagram
\beq\la{nag st}
\xymatrix{
X \ar[d]_-j \ar[rrrr]^f  \ar[rrrrd]^(.75){h} &&&& Y 
\\
Z \ar[rrrr]^-q \ar[urrrr]^(.65){g} && &&  W \ar[u]_-u
}
\eeq
where:   $(j,g)$ is a Nagata-Deligne completion (\ci{conrad}) of the morphism $f$, i.e. $j$ is  an open and dense immersion and $g$ is proper;
we can and do choose $Z$ to be integral; $(q,u)$ is the Stein Factorization 
\cite[\href{https://stacks.math.columbia.edu/tag/03H0}{03H0}]{stacks-project} of $g$, so that $q$ has geometrically connected fibers and $u$ is finite. Note that $W$ is integral,   that
$g, q$ are  surjective, and that $u$ is finite and surjective. 

By \ci[Lemma 4.4.2]{de-ha-li}  (this is stated for the case when $W$ and $Y$ are varieties over an algebraically closed field; however 
the proof works also in our situation, where $W$ is integral, and $Y$ is integral and noetherian), there is a canonical factorization 
\beq\la{can ftz}
\xymatrix{
u= s\circ i: W \ar[r]^-i &  W' \ar[r]^-s & Y,
}
\eeq with $i$ finite radicial (hence a universal monomorphism) and surjective  (hence a  a universal bijection),  and $s$ is finite, surjective, separable (\ci[Definition 4.4.1]{de-ha-li}) and generically \'etale.

Our goal is to prove that $u$ is bijective, i.e. that $s$ is bijective. If this were the case, then we would be done as follows. 
Since $Y$ is quasi-compact, and $g$ is proper, by \cite[\href{https://stacks.math.columbia.edu/tag/04XU}{04XU}]{stacks-project}, we have that $Z$ is quasi-compact. 
Therefore the closed subspace $Z\setminus X$ is also quasi-compact, thus, by \cite[\href{https://stacks.math.columbia.edu/tag/005E}{005E}]{stacks-project},  if $Z\setminus X$ is nonempty, then $Z\setminus X$ has a closed point.  
Now let $y\in Y$ be a closed point and let $w \in W$ be its unique pre-image via $u$.
Then $j(f^{-1}(y))$ is open in $q^{-1}(w)$,
but it is also closed since $X_y=f^{-1}(y)$ is proper over $y$ by assumption. The connectedness of $q^{-1}(w)$ implies that
set-theoretically $j(f^{-1}(y))$ equals  $q^{-1}(w)$, i.e., $j$ induces a bijection $f^{-1}(y) \to g^{-1}(y)$.
Since this is true for every closed point $y\in Y$, and since $g$ is proper, we see
that $j$ induces a bijection between the closed points of $Z$ and the ones of $X$.
Therefore $Z\setminus X=\emptyset$, thus $Z=X$, i.e., our contention that $f$ is proper holds true.

We are left with proving that $s$ is bijective. Note that the formation of the canonical factorization (\ref{can ftz}) is compatible with restrictions to open subsets in $Y.$ Since $W$ and $Y$ are integral, $Y$ is normal, $s$ is finite, and a finite birational morphism from an integral scheme to an integral and normal scheme is an isomorphism \cite[\href{https://stacks.math.columbia.edu/tag/0AB1}{0AB1}]{stacks-project},
it is enough to show that $s$ is an isomorphism over a Zariski dense open subset $U$ of $Y.$ The remainder of the proof is dedicated to proving this assertion.

Note that $h$ is dominant. 
Since the  image $\im (h)$  is constructible and dense,  it contains a Zariski dense open subset $V \subseteq W.$
Then $u(W \setminus V)$ is a proper closed subset of $Y$ with open and dense complement  which we denote by $U$. Then $h$ is surjective over the open dense $u^{-1}(U)$. 
It follows that,  in view of proving that $s$ is an isomorphism, it is enough (as seen above)  to prove it when $h$ is surjective, which we assume hereafter.

For any closed point $w\in W$, by the connectedness of the fibers of $f$ and the surjectivity of $h$, we have that, set-theoretically, $h(f^{-1}(u(w)))$ is contained in the same connected component of $u^{-1}(u(w))$ as $w$, and also contains $u^{-1}(u(w))$. Therefore, as a scheme,  $u^{-1}(u(w))$ is connected and it is  finite over the residue field of $u(w)$.
By \cite[\href{https://stacks.math.columbia.edu/tag/00KJ}{00KJ}]{stacks-project}, we have that $u^{-1}(u(w))$, as a set, is a singleton.
We thus have that $h^{-1}(w)=f^{-1}(u(w))$.
As seen above, $j(h^{-1}(w))$ is then open and closed in the connected $q^{-1}(w).$

As seen above, this implies that $j$ is an isomorphism and then $f=g$ is proper with geometrically connected fibers. 
Since geometrically connected schemes are universally connected \cite[\href{https://stacks.math.columbia.edu/tag/054N}{054N}]{stacks-project}, we have that $s$ is separable and universally bijective. 
By \cite[\href{https://stacks.math.columbia.edu/tag/01S4}{0154}]{stacks-project}, we have that $s$ is an isomorphism (recall we shrunk $Y$ to $U$).
But then $s: W' \to Y$ is an isomorphism over $U$, and this concludes our proof.
\end{proof}

\begin{rmk}\la{need hyp}
The case when  $f$ is the normalization of a nodal curve, with a point removed from the domain, and $m$ is the identity, shows that normality  cannot be dropped from the list of assumptions in Proposition \ref{ad hoc pr}.
The case when $X$ the disjoint union of a line and a line without the origin, with $f$ the natural morphism to  a line, with $m$ the identity, shows that the irreducibility of $X$ cannot be dropped. The case of  $f$ the square map ${\Gm}_{,\field} \setminus \{-1\} \to {\Gm}_{,\field}$ (say $char (\field) \neq 2$)  
and $m$ the identity, shows that the connectedness of the fibers cannot be dropped.
\end{rmk}

\begin{pr}[{\bf (Hodge-Hitchin is proper)}]\la{hh is proper}
Let $C/\field$ be a smooth curve (\S\ref{rino}) over an algebraically closed field $\field$ of characteristic $p>0.$
Let $\ov{d} \in \zed$ and assume that ${\rm g.c.d.} (r, \ov{d}p)=1.$
The  Hodge-Hitchin morphism $h_{Hod}(C;r,\ov{d}p)$ (\ref{mz2})  is projective.
\end{pr}
\begin{proof} 
We drop some decorations.
Since domain and target are quasi projective, it is enough to prove the properness of $h_{Hod}.$
Recall (\ref{d100m}) that for $t=0 \in \Aa^1_k$,  the morphism $h_{Hod,0}$ is the Hitchin morphism composed with
the relative Frobenius $Fr_A$ (a universal homeomorphism) of the Hitchin base.  In view of (\ref{hh91}), for  $t \in {\Gm}_{,\field}(\field),$ the morphism
$h_{Hod,t}$ is isomorphic to the morphism $h_{dR}.$ 

We wish to apply  Proposition \ref{ad hoc pr} with $m\circ f: X \to Y \to T$ given by $\tau = {\rm pr}_{\mathbb A^1_\field} \circ h_{Hod}$
(\ref{mz2}).
In order to do so, we need to verify that the  hypotheses (1-4) are met in our setup. 

(1) is clear.  As to (2), we argue as follows.
By Proposition \ref{t is smooth},  $X:= M_{Hod}$ is integral, $Y:= A(C^{(1)})\times \mathbb A^1_\field$ is integral and normal (in fact nonsingular).

As to (3), we need to establish the surjectivity of $f=h_{Hod}$, and the geometric connectedness of its fibers.
The morphism $h_{Hod}$ is surjective; in fact, according to the proof of Lemma
\ref{nonempty}:  over the origin $0 \in \Aa^1_\field,$ the Hitchin morphism is 
 surjective (and proper), and so is $Fr_A$;
over ${\Gm}_{,\field}$ the surjectivity follows from the trivialization  (\ref{hh91}) and the  surjectivity of
 (the proper) $h_{dR}$.  

Let us argue that the morphism  $h_{Hod}$ has geometrically connected fibers. It is enough to prove that for every closed point $t \in \Aa^1_\field$,
$h_{Hod,t}$ has geometrically connected fibers. In view of the trivialization (\ref{hh91}), we need to prove this only for $t=0$, where we get the  Hitchin morphism composed with $Fr_A$,
and for $t=1$, where we get the Hitchin-de Rham morphism. 
The  fibers of the Hitchin morphism  are geometrically connected
by Zariski Main Theorem (so that so are the fibers of its composition with $Fr_A$): domain and target are nonsingular integral and the general fibers are connected (Jacobians of nonsingular connected projective spectral curves; cf \ci[Prop. 3.6]{be-na-ra}). 
As seen in the proof of Lemma \ref{nonempty}, the fibers of the de Rham-Hitchin morphism for $C$  in degree $\ov{d}p$,
are isomorphic to the fibers of the Hitchin morphism for the Frobenius twist $\FC$ in degree $\ov{d}$, and are thus also geometrically connected.
This concludes the verification that hypothesis (3)  holds.

The morphisms $f_t= h_{Hod,t}$ are:  for $t=0$ (\ref{d100m}),  the Hitchin morphism  composed with $Fr_A$;
for $t=1$ (\ref{d11}), the   de Rham-Hitchin morphism; for $t\neq 0$, isomorphic to the de Rham-Hitchin morphism
in view of the trivialization (\ref{hh91}).
The Hitchin morphism is   proper (\ci{fa, nitsure, si mo II}).  The relative Frobenius morphism $Fr_A$ is finite,
hence proper.
The de Rham-Hitchin morphism 
is proper by \ci[Cor. 3.47]{gr-2016}.
It follows  that hypothesis (4) holds as well.

We are now in the position to apply  Proposition \ref{ad hoc pr} and conclude.
\end{proof}

\section{Cohomological Simpson Correspondence  in positive characteristic}\la{the mainz}$\;$

{\bf Assumptions in \S\ref{the mainz}.} In this section, we place ourselves in  the following special case of the set-up in \S\ref{rino}:
$C= C/\field$ is  a smooth curve over an algebraically closed field $\field$ of positive characteristic $p>0$,
the degree $d=\ov{d} p$ is an integer multiple of the characteristic and ${\rm g.c.d.} (r,d)=1.$ 
At times, we drop some decorations.

{\bf The three main results in this \S\ref{the mainz}.}
We prove three main results. Theorem \ref{zmain}: a canonical cohomological version of the Simpson correspondence
between the moduli spaces of Higgs bundles and of connections. The perhaps surprising Theorem \ref{central fiber}
yielding a canonical  isomorphism between the cohomology rings of the moduli space of  connections and the moduli space of connection with nilpotent $p$-curvature tensor.  The perhaps even more surprising,
 especially when compared with the well-known and evident  ``additive periodicity" (\ref{gkgk4}), ``$p$-multiplicative periodicity,"
Theorem \ref{int cons} involving  the Frobenius twists of a curve.

{\bf The perverse Leray filtrations we use.}
The \'etale cohomology ring $H^* (M_{dR}(C), \oql)$  is filtered by 
 the perverse Leray filtration   $P^{h_{dR}(C)}$  (\ref{def plf}), 
associated with  the de Rham-Hitchin morphism $h_{dR}(C)$ (\ref{d11}).  
Similarly, we have 
 the perverse Leray filtration $P^{h_{Dol}(C)}$  (\ref{def plf}) on $H^* (M_{Dol}(C), \oql)$, associated with 
 the Hitchin morphism $h_{Dol}(C)$  (\ref{d100m}).  
 
 Since the relative Frobenius morphism $Fr_A$ in (\ref{d100m}) is finite, in view of 
(\ref{comp pt}),  we have that
\beq\la{the2co}
\mbox{
 $P^{h_{Dol}(C)}\, = \, P^{h_{Hod,0}(C)}$ \;\;on $H^*(M_{Dol}(C), \oql)$.}
\eeq

\subsection{A cohomological Simpson Correspondence in positive characteristic}\la{11scpc}$\;$

Recall that the moduli space $M_{dR}$ on the  r.h.s. of the forthcoming (\ref{maineq2}) is empty in characteristic zero.
The $M_{Dol}$ on the l.h.s. is non empty and lifts to characteristic zero.

\begin{tm}
[{\bf (Cohomological Simpson Correspondence $char(\field)=p>0$, I)}]
\la{zmain}
Let $C/\field$ and ${\rm g.c.d.}(r,d=\ov{d}p)=1$ be as in the beginning of \S\ref{the mainz} above. 

\n
There is a natural filtered isomorphism of  cohomology rings
\begin{equation}\label{maineq2}
\left( H^*(M_{Dol}(C;r,\ov{d}p), \oql),  P^{h_{Dol}} \right) \simeq
\left( H^*(M_{dR}(C;r,\ov{d}p), \oql), P^{h_{dR}} \right).
\end{equation}
\end{tm}
\begin{proof}
We drop some decorations. Recall that:  the Hodge-Hitchin morphism at $t=1 \in \Aa^1_\field$ coincides with the de Rham-Hitchin morphism,
i.e. 
 $h_{Hod,1}= h_{dR}(C)$ (\ref{d11}); the  Hodge-Hitchin morphism at $t=0 \in \Aa^1_\field$ coincides with the composition
 of $Fr_A\circ h_{Dol}$ (\ref{d100m}).
We apply the  formalism of vanishing and nearby cycles recalled in \S\ref{rem vn}, to the two morphisms
\beq\la{oq12}
\xymatrix{
\tau: M_{Hod}(C) \ar[rr]^-{h_{Hod}} && A(C^{(1)})\times \mathbb A^1_\field \ar[rr]^-{\rm \pi:= pr_2} && \mathbb A^1_\field,
\\
\sigma: M_{dR}(C) \times \mathbb A^1_\field \ar[rr]^-{h_{dR}\times {\rm Id}_{\mathbb A^1_\field}}  &&   A(C^{(1)}) \times \mathbb A^1_\field \ar[rr]^-{\rm \pi:= pr_2} &&  \mathbb A^1_\field.
}
\eeq
Note that the morphism  $\tau$ and $\s$ share the second link $\pi$.

We  take $S$ to be a strict Henselianization of the spectrum of the completion of the local ring of the point $i: 0 \to  \mathbb A^1_\field.$
By Lemma \ref{nonempty}  and Proposition \ref{t is smooth}, the morphisms $\s$ and $\tau$ are smooth.  
In particular, $\phi_\tau (\oql)=0$ and $\phi_\sigma (\oql)=0$ (cf. Fact \ref{fvz}),
 so that we have $\psi_\tau (\oql)=\oql$ on $M_{Dol}(C)=M_{Hod,0}(C)$, and $\psi_\sigma (\oql)=\oql$ on $M_{dR}(C).$

By Proposition \ref{hohiisproper}, the morphisms $h_{Hod}$ is proper.
Since the de Rham-Hitchin morphism $h_{dR}$ is proper,  the morphism $h_{dR}\times {\rm Id}_{\mathbb A^1_\field}$ is proper.
In particular, we have natural isomorphisms  in $D^b_{c}(A(C^{(1)}) \times 0, \oql)$ stemming from the proper base change isomorphisms
($i^* h_* = h_* i^*, \psi h_* = h_* \psi$)
\begin{align}\la{pi2}
\begin{split}
\xymatrix{
i^*  {h_{Hod}}_* \oql \simeq {h_{Hod,0}}_*\oql, &
i^* ( h_{dR}\times {\rm Id}_{\mathbb A^1_\field} ) _* \oql  \simeq {h_{dR}}_*\oql,
\\
{h_{Hod,0}}_*\oql  \simeq \psi_\pi  ({{(h_{Hod}}_{| {\Gm}_{,\field}})_*} \oql ), &
{h_{dR}}_*\oql \simeq \psi_\pi  (h_{dR}    \times {\rm Id}_{{\Gm}_{,\field}}  ) _*\oql.
}
\end{split}
\end{align}

By the trivializing isomorphism  (\ref{hh91}), we 
have a natural isomorphism between the two terms of type $\psi_\pi$ in (\ref{pi2}).
We thus  have a natural isomorphism in $D^b_{c} (A(C^{(1)}))$
\beq\la{to ex}
{h_{Hod,0}}_*\oql \simeq {h_{dR}}_*\oql.
\eeq

Ignoring the ring structure: the statement in   cohomology follows  by taking cohomology in (\ref{to ex}); the filtered refinement, follows from (\ref{the2co}).

As to the ring structure, we argue as follows.

Recall that to obtain the isomorphism (\ref{to ex}) we need to pass through three types of morphisms: firstly, the morphisms induced by $i^*\to \psi$; 
secondly, the morphisms induced by the base change morphism; and lastly, the morphism induced by the trivializing isomorphism (\ref{hh91}).
We need to show that all three types of morphisms above preserves cup products.

We now consider the first type. Note that the vanishing cycle functor preserves cup products (see e.g. \cite[\S4.3]{illusie}). Upon taking cohomology on $\mathbb{A}^1$, the morphism $i^*\to \psi$ induces the specialization morphism on stalks as defined in \cite[\href{https://stacks.math.columbia.edu/tag/0GJ2}{0GJ2}]{stacks-project}. By the description of the specialization morphism in terms of pulling back sections via $\bar{j}^*$ (\ref{jbar}) as in \cite[\href{https://stacks.math.columbia.edu/tag/0GJ3}{0GJ3}]{stacks-project}, we see that the morphism $i^*\to \psi$ preserves cup products.

To show that the second type of morphisms preserve cup products, we are reduced to show that a base change morphism of the form $i^* h_*\to h_* i^*$ preserves cup product. We can write the base change morphism as the composition $i^*h_*\to i^* h_* i_* i^*\xra{\sim} i^*i_* h_* i^*\to h_*i^*$, where the first morphism is induced by the unit morphism $id\to i_*i^*$ and the last by the counit $i^*i_*\to id$. It is easy to check that both preserve cup products.

Finally, the trivializing isomorphism (\ref{hh91}) is induced by an actual isomorphism (\ref{prfr}) of varieties, and it does preserves cup products.
\end{proof}

\begin{rmk}[\bf (Weights)]
If  the curve $C/\field$ is obtained by extensions of scalars from a curve over a finite field, then the isomorphism (\ref{maineq2}) is compatible with the Frobenius weights (see \ci[Thm. 6.1.13]{weil2}).
The same is also true for  the isomorphisms in the forthcoming Theorems \ref{central fiber}, \ref{int cons}, \ref{via pb}, \ref{vvooqq}, \ref{e4}, \ref{c cop} and \ref{c cop p}.
\end{rmk}

\subsection{Cohomology ring of the space of  connections with nilpotent \texorpdfstring{$p$-curvature}{p}}\la{fgt5}$\;$

The following Theorem \ref{central fiber}  is a somewhat unexpected and surprising consequence of Theorem \ref{zmain}.
This is because its analogue (\ref{re mo iso}) for the Dolbeault moduli space is well-known to experts and proved using the $\Gm$-equivariance
and properness
of the Hitchin morphism, whereas in the de Rham case, there is no natural non-trivial $\Gm$-action. In particular,  even ignoring the
filtrations and the ring structure, there seems to be no clear a priori reason why the isomorphism (\ref{nu2}) should hold additively.

{\bf The fiber $N_{dR}$.}
Let $C/\field$ and ${\rm g.c.d.}(r,d=\ov{d}p)=1$ be as  in \S\ref{the mainz}.
Let $N_{dR}(C;r,\ov{d}p)$ be the fiber over the origin $i_{o(1)}: o(1) \to A(\FC;r)$ of the de Rham-Hitchin morphism  $h_{dR}(C;r,d)$ (\ref{d11}). This is the moduli space of  those stable stable connections of rank $r$ and degree $d$  with nilpotent $p$-curvature Higgs field.
Let us drop $r$ and $d$ from the notation.

{\bf The filtration $P_{N_{dR}}$ on $H^*(N_{dR}, \oql)$.}
The inclusion of this fiber induces the  cohomology ring homomorphism $i_{o(1)}^* : H^*(M_{dR}) \to H^*(N_{dR}).$ The perverse $t$-structure  on $A(\FC)$ induces
a filtration $P$ on the cohomology of the fiber $\w{M_{dR}}$ of $h_{dR}$ over the strict localization $\w{o(1)}$  of $o(1).$ By proper base change,
restriction induces a cohomology ring  isomorphism $H^* (\w{M_{dR}},\oql) \simeq H^* (N_{dR},\oql)$, and, by transport of structure,
the latter cohomology group inherits the filtration,  denoted by $P_{N_{dR}}$, from the former (not to be confused with the perverse Leray filtration
induced by the morphism $N_{dR} \to o(1)$, which is trivial-shifted by the degree in  each cohomological degree).
We thus have that restriction induces a filtered morphisms of cohomology rings 
\beq\la{nu1}
\xymatrix{
i_{o(1)}^*: H^* \left(M_{dR}(C;r,\ov{d}p),\oql), P^{h_{dR}}\right) \ar[r] &
\left(H^* (N_{dR}(C;r,\ov{d}p),\oql), P_{N_{dR}}\right).
}
\eeq 

\begin{rmk}\la{spltz}
The Decomposition Theorem \ci[Thm. 6.2.5]{bbdg} (stated over $\comp$, but  valid over any algebraically closed ground field),  and  the construction of $P_{N_{dR}}$,  imply that one can split the perverse filtrations $P^{h_{dR}}$ and $P_{N_{dR}}$ compatibly with the restriction morphism $i_{o(1)}^*$, i.e. this latter is a direct sum morphism for the two filtrations split into direct sums. In particular, if $i_{o(1)}^*$ is an isomorphism, then it is a
filtered isomorphisms. Recall that isomorphisms that are filtered morphisms, may fail to be filtered isomorphism.
By replacing $``dR" with ``Dol"$, we see that the same holds for $P^{h_{Dol}}$ and $P_{N_{Dol}}$,
where $N_{Dol}$ is the fiber over $o \in A(C)$ of the Hitchin moprhism $h_{Dol}: M_{Dol}(C) \to A(C).$ 
\end{rmk}

Recall our assumptions \S\ref{the mainz}: 
 $C/\field$, $char (\field)=p>0,$  and ${\rm g.c.d.}(r,d=\ov{d}p)=1$.

\begin{tm}[{\bf (The cohomology ring of $N_{dR}$)}]\la{central fiber}
The morphism (\ref{nu1})   is a filtered isomorphism of cohomology rings
\beq\la{nu2}
\xymatrix{
i_{o(1)}^*: H^* \left(M_{dR}(C;r,\ov{d}p),\oql), P^{h_{dR}}\right) \ar[r]^-\simeq &
\left(H^* (N_{dR}(C;r,\ov{d}p),\oql), P_{N_{dR}} \right).
}
\eeq
\end{tm}
\begin{proof}  We drop many decorations. We start by proving the forthcoming and seemingly well-known (cf. \ci[Thm. 1, for example]{he2015}) (\ref{re mo iso}), the proof of which   remains valid without restrictions
on rank, degree, nor  characteristic of the ground field.

Let $N_{Dol}$ be the fiber of the Hitchin morphism $h_{Dol}: M_{Dol} \to A(C)$ over the origin $i_o: o \to  A(C).$ 
The complex ${h_{Dol}}_* {\oql}_{M_{Dol}}$ is $\Gm$-equivariant for the natural  $\Gm$-action on $A(C)$
 (cf. the paragraph following (\ref{d100m})).
Since $h_{Dol}$ is proper, proper base change (pbc), coupled with  \ci[Lemma 4.2]{de-mi-mu}, implies that  the adjunction morphism
\beq\la{adj mo}
\xymatrix{
{h_{Dol}}_* {\oql}_{M_{Dol}}  \ar[r]  & {i_o}_* i_o^* {h_{Dol}}_* {\oql}_{M_{Dol}} \ar[r]^-{\simeq}_-{\rm pbc}  &{i_o}_* {h_{Dol}}_*
{\oql}_{N_{Dol}},
}
\eeq
induces an isomorphism.  By taking cohomology, this morphism induces the  restriction morphism in cohomology,
which is thus 
an isomorphism of cohomology  rings
\beq\la{re mo iso}
\xymatrix{
i_o^* : \left( H^*(M_{Dol}, \oql), P^{h_{Dol}} \right)    \ar[r]^-\simeq  & \left(H^*(N_{Dol},\oql), P_{N_{Dol}}\right).
}
\eeq
In view of Remark
\ref{spltz},
 this is also   a filtered isomorphism.

Recall  diagrams  (\ref{d11}) and (\ref{d100m}).
 Let $i_{o(1)}: o(1) \to A(C^{(1)})$ be the origin, so that $N_{dR}$ is the corresponding fiber of the de Rham-Hitchin morphism
$h_{dR}: M_{dR} \to  A(\FC)$. Let $Fr_{A(C)}^{-1}(o(1))$ be the fiber of $Fr_{A(C)}$ over $o(1);$ it is supported at
the origin $o \in A(C).$ The fiber  $h_{Dol}^{-1}(o)=N_{Dol}$ is a closed subscheme of the fiber $
[{N_{Dol}}]:= h_{Dol}^{-1} (Fr_{A(C)}^{-1}(o(1))) =  h_{Hod,0}^{-1} (o(1))$, and these two fibers have the same reduced structure, hence the same cohomology ring
(more precisely, identified by pull-back). In view of the isomorphism (\ref{re mo iso}), we have isomorphisms of cohomology rings $H^*(M_{Dol}, \oql) \simeq
H^*([{N_{Dol}}],\oql) \simeq H^*(N_{Dol},\oql)$.

By applying the adjunction morphism of functors  ${\rm Id} \to i_{o(1) *} i_{o(1)}^*$ to the isomorphism
(\ref{to ex}), which we recall induces an isomorphism of cohomology rings, we obtain the following 
commutative diagram of  morphisms of cohomology rings, 
where the vertical arrows are the restriction morphisms of cohomology rings,
and with the indicated three isomorphisms of cohomology rings
\beq\la{rti5}
\xymatrix{
H^*(M_{Dol}, \oql)  \ar[r]^-\simeq  \ar[d]^-\simeq& H^*(M_{dR}, \oql) \ar[d]
\\
H^*([{N_{Dol}}], \oql)  \ar[r]^-\simeq & H^*(N_{dR}, \oql).
}
\eeq
It follows that the fourth unmarked vertical arrow on the rhs, which  is the restriction morphism $i_{o(1)}^*$ in (\ref{nu2}),
is  an isomorphism of cohomology rings.

Finally, since we now know that $i_{o(1)}^*$ is an isomorphism,  and a filtered morphism (\ref{nu1}), Remark \ref{spltz} implies that $i_{o(1)}^*$ is a filtered isomorphism as predicated in (\ref{nu2}).
\end{proof}

\subsection{Cohomology ring of moduli spaces for a curve and its Frobenius twist}\la{c and c1}$\;$

Note that in the construction of the Frobenius twist $C^{(1)}:= C \times_\field \field$ of a $\field$-scheme, we can replace 
the field automorphism $fr_\field: \field \stackrel{\sim}\to \field$, $a \mapsto a^p$ with any of its integer powers and  obtain,
for  every integer $n \in \zed$,  the $n$-th iterated Frobenius twist  $C^{(n)}$  of $C.$  The curve $C$ and all its Frobenius twists have the same genus.

The following  ``multiplicative periodicity" result, involving the characteristic $p$ as a factor and the  Frobenius twists of $C$, is a simple, yet remarkable consequence of Theorems \ref{zmain},   \ref{central fiber}, and  
\cite[Cor.\;3.28]{gr-2016}. It allows to prove the forthcoming ``multiplicative periodicity result 
Theorem \ref{e4}, involving only the  curve $C$, and not its Frobenius twists.

Recall our assumptions \S\ref{the mainz}: 
 $C/\field$, $char (\field)=p>0,$  and ${\rm g.c.d.}(r,d=\ov{d}p)=1$. 
 
\begin{tm}[{\bf ($p$-Multiplicative periodicity with Frobenius twists)}]\la{int cons}$\;$

\n
Let $d=\w{d}p^m$, with $m \geq 0$ maximal.
We have  canonical isomorphisms of cohomology rings
\begin{equation}\la{km33}
\xymatrix{
H^*\left(M_{Dol}\left(C;r , \w{d}p^m \right),\oql\right)
&\cong &
    H^*\left(M_{Dol} \left(C^{(m)}; r, \w{d}\right),\oql\right),
    \\
H^*\left(M_{Dol}\left(C^{(-m)}; r, \w{d}p^m \right),\oql\right)
&\cong &
    H^*\left(M_{Dol} \left(C ;r , \w{d}\right),\oql\right);
    }
    \end{equation}
similarly,  if we replace $\w{d}$ with $\ov{d}$.

These isomorphisms are  filtered  isomorphisms for  the respective perverse Leray filtrations.
\end{tm}
\begin{proof} We prove the statements for $\w{d}$. The same line of argument applies to $\ov{d}$.

Since $C$ can be any projective nonsingular curve of a fixed genus, by using  Frobenius twists,
we see that the two assertions are equivalent to each other. It is enough to prove  the one in the top row.
The case $m=0$ is trivial.
A simple induction on $m$ shows that  it is enough to prove the top row when  $m=1$.

We use the notation in the proof of  Theorem \ref{central fiber}. 
We recall that the two morphisms $h_{Dol}: M_{Dol}(\FC; r, \ov{d}) \to A(\FC;r)$ and $h_{dR}: M_{dR}(C; r, \ov{d}p) \to A(\FC;r)$
are \'etale locally equivalent over their common target $A(\FC;r)$; see \cite[Cor.\;3.28,\;Lemma 3.46]{gr-2016}.  This immediately implies that 
the two fibers over the origin $N_{Dol}(\FC; r, \ov{d})$ and $N_{dR}(C, \ov{d}p)$ are isomorphic as $\field$-varieties.
As in the proof of \cite[Cor.\;3.45]{gr-2016}, we choose a distinguished isomorphism between $h_{Dol}$ and $h_{dR}$ over an \'etale neighborhood $U$ over the origin of $A(C^{(1)},r).$ By taking the fiber of this isomorphism over the origin of $A(C^{(1)},r)$, we obtain
a cohomology ring isomorphism $\nu: H^*(N_{Dol})\xra{\sim} H^*(N_{dR})$. By the very construction of the filtrations $P_{N_{Dol}}$ and $P_{N_{dR}}$ in 
\S\ref{fgt5}, the isomorphism $\nu$ is filtered for $P_{N_{Dol}}$ and $P_{N_{dR}}$.

By invoking the appropriate results in parentheses, we have the following chain of canonical ring  filtered isomorphisms
(filtrations are omitted for typographical reasons).

\beq\la{kjqw}
\xymatrix{
H^*\left(M_{Dol} \left(\FC, \w{d}\right),\oql\right)  & \stackrel{(\ref{re mo iso})}\cong &
H^*\left(N_{Dol} \left(\FC, \w{d}\right),\oql\right) 
\\
&
 \stackrel{\mbox{\cite[3.28 and 3.46]{gr-2016}}}\cong
 & H^*\left(N_{dR} \left(C, \w{d}p\right),\oql\right) 
\\
& \stackrel{(\ref{nu2})}\cong & H^*\left(M_{dR} \left(C, \w{d}p\right),\oql\right)
\\
& \stackrel{(\ref{maineq2})}\cong  & H^*\left(M_{Dol} \left(C, \w{d}p\right),\oql\right).
}
\eeq
This proves the top row in (\ref{km33}).  
\end{proof}

\section{Cohomological equivalence of Hodge moduli spaces of curves}\la{omq2}$\;$

In \S\ref{the mainz}, we worked with a fixed curve $C/\field$ over an algebraically closed field $\field$ of characteristic $p>0$,
and, under certain conditions on $r$, $d$ and $p$,  we have  used the family $\tau: M_{Hod}(C) \to \Aa^1_\field$
to relate (the cohomology of)  $M_{Dol}$ and $M_{dR}$ in the same degree (Theorem \ref{zmain}).
We have also been able to relate    $M_{Dol}(C) $  and $M_{Dol}(C^{(-n)})$ when the degrees differ by a factor $p^n$
($p$-multiplicative periodicity with Frobenius twists Theorem \ref{int cons}).

In this section, we build on these results and, under certain conditions on $r$, $d$ and $p$,
we relate (the cohomology of) $M_{Dol}$ with fixed degree  for different curves  of the same genus
(Theorem \ref{vvooqq}), 
and with  different degrees  (Theorem \ref{e4}) differing by a factor power of $p$ for the same curve (hence for different curves of the same genus). 

This latter result is then lifted to characteristic zero, where, coupled with the Dirichlet Prime Number Theorem,
relates (the cohomology of) $M_{Dol}$ in different degrees (Theorem \ref{c cop}) for a curve (hence for different curves).
 The existence of such an isomorphism in cohomology is known, but the compatibility of the perverse filtrations is new.
 
This  result in characteristic zero is  then specialized back to characteristic $p>r$ (Theorem \ref{c cop p}), where it is new.

The main technical tool employed in this \S\ref{omq2}, and that has not been used in proving the results in \S\ref{the mainz}, is part of the compactification/specialization  package developed   \ci{de-2021} and generalized in part in \ci{de-zh B}. We summarize what we need in Proposition 
\ref{sp cptz}. In order to have access to this package, we need to establish the smoothness (Proposition \ref{rel smooth}) and  the properness
(Proposition \ref{prpbz})  of the morphisms we employ.

\subsection{Relative moduli spaces: smoothness and properness}\la{kj11}$\;$

In this subsection, we prove 
Proposition \ref{rel smooth}, i.e. the smoothness of the Hodge-moduli space $M_{Hod}(\csub) $ for  a projective smooth family
$\csub$ of curves  over a nonsingular base  curve $B$.
We also prove
Proposition \ref{prpbz}, i.e. the properness of the Hodge-Hitchin morphism 
for said family. 
 These two results are the relative-version over a  base curve  of Theorems \ref{t is smooth} and \ref{hh is proper}.
They are used in the proof of Theorem \ref{vvooqq}.  In fact, we only  need the specialization of these two results
to the case of the Dolbeault moduli space, where the properness of the Hitchin morphism  is well-known, while the smoothness assertion seems new, at least in positive characteristic.

\begin{pr}
[{\bf (Smoothness  of moduli over  a base)}]\label{rel smooth}
Let $\csub$ be a smooth curve (\S\ref{rino})  over a reduced  base  $\base$.

The following morphisms are smooth surjective and quasi projective
\ben
\item
$\ah_\base: M_{Hod}(\csub,r,\ov{d}p) \ra \base$; here ${\rm g.c.d.} (r,d)=1.$

\item 
$\beta_\base: M_{Dol}(\csub,r,d)\ra 0_{\base}\cong \base$; here, ${\rm g.c.d.} (r, d)=1;$

\item
$\tau_\base : M_{Hod}(\csub,r,d)\ra \Aa^1_{\base};$ here,  $\brg $ is an algebraically closed field of characteristic $p>0$,  and ${\rm g.c.d.} (r, \ov{d}p)=1;$

\item
$\gam_\base : M_{dR}(\csub,r,d)\ra 1_{\base}\cong \base$; here,  $\brg $ is an algebraically closed field of characteristic $p>0$,  and ${\rm g.c.d.} (r, \ov{d}p)=1;$
\een
Moreover:  if $\base$ is integral, then the domains of these morphisms are  integral; 
if $\base$ is nonsingular, 
then the domains are nonsingular.

\end{pr}
\begin{proof} Surjectivity can be checked after base change via geometric points $b\to \base$, in which case it follows from Proposition
\ref{t is smooth}.
The quasi projectivity follows from the fact that the moduli spaces are quasi projective over $\base.$
Note that parts (iii) and (iv) fail if we do not assume that $d$ is a multiple of $p$, for then $M_{dR}$ is empty.
Part (i)  implies parts (ii) and (iv) via the base changes   $0_\base, 1_\base \to \Aa^1_\base$. 
Part (i) coupled with the flatness of the morphisms $\tau_b$ at the geometric points of $\base$ (Proposition \ref{t is smooth})
implies part (iii) in view of  \ci[IV.3, 11.3.11]{EGA4.3}, which states that a $B$-morphism $f:X\to Y$ is flat if $X$ is flat over $B$ and the base change of $f$ to each point $b\in B$ is flat.

It follows that we only need to prove part (i).
The proofs of (i) follow the same thread as the proof of smoothness in Theorem \ref{t is smooth}. As the proof we are about to give shows,
we are really implicitly proving (ii) as we prove explicitly (i).

{\bf Proof of part (i).}
Since the fibers of $\alpha_\base$ are smooth (Proposition \ref{t is smooth}), it is enough to prove the flatness
of the  locally finitely presented morphism $\alpha_\base$. 
By  the valuative criterion of flatness
\ci[IV.3, 11.8.1]{EGA4.3}, 
we  can  replace our  $\base$ with the spectrum  $A$ of a DVR mapping to $\base$.
The proof that  $\alpha_A$ is flat is very similar to the proof
of Proposition \ref{t is smooth}. Note that in order to use the valuative criterion of flatness, we need the assumption that $B$ is reduced.

In the present context, the only point that requires a different proof is the analogue of CLAIM 1 in the proof of said proposition: it is enough 
to  exhibit an Higgs bundle   on the  curve $X_A/A$ over the DVR $A$. 
In order to conclude the proof of part (i) it is thus sufficient to prove the forthcoming CLAIM 1A.
Let $a$ and $\alpha$ be the closed and  open points of $A.$

{\bf CLAIM 1A:} {\em   We have $\ov{M_\alpha }  \cap M_a \neq \emptyset.$}

By the BNR correspondence \cite[Prop.\;3.6]{be-na-ra} for smooth spectral curves:  (a line bundle of the appropriate degree
on a smooth  degree $r$ spectral curve $S/A$)  $\mapsto$ (a stable Higgs bundle of the appropriate degree on the curve $X_A/A$).

If $\ov{M_\alpha }$ and $M_a$ were disjoint, then they would stay disjoint after any base change $Z\to A$ covering $a$.
It  is thus enough to show that we can extend  any line bundle on any  smooth spectral curve $S_a$ over $C_a$ to a  line bundle on  a smooth spectral curve $S_A$ over  $C_A$, possibly after an \'etale base change $Z\to A$ covering $a.$

Let $u: \m{S}\to A(C_A/A, \omega_{C_A/A})$ be the universal spectral curve of degree $r$ for the family $C_A/A.$
 Since the universal spectral curve is flat over the Hitchin base, and the Hitchin base is flat over $A,$
the universal curve is flat over $A.$
By using the Jacobian criterion in connection with the polynomial expression for the equations of spectral curves, we see that
$\m{S}/A$, being flat, is smooth. Then, since for every geometric point $a$ on $A$ the fiber $\m{S}_a$ is nonsingular integral,  we see that $\m{S}$ is integral.   The morphism $u$ is not smooth, but since general spectral curves are nonsingular  --this is true over both points $a,\alpha \in A$--,  we have that
there is an open and dense subset $U \subset A(C_A/A, \omega_{C_A/A})$ over which $u$ is smooth and such that the resulting morphism $U \to A$ is smooth and surjective. Moreover, the geometric fibers of $\mathcal{S}$ over $U$ are nonsingular integral. 
By \cite[Thm.\;9.4.8, Prop.\;9.5.19]{kleiman},  the Picard scheme $Pic_{\mathcal{S}_{U}/U}$ exists as a smooth group scheme over $U$ which is separated and  locally of finite type  over $U$. Note that $Pic_{\mathcal{S}_{U}} /U$ is smooth and surjective.
In particular, $Pic_{\mathcal{S}_{U}} /A$ is smooth and surjective.
By \cite[\href{https://stacks.math.columbia.edu/tag/054L}{054L}]{stacks-project}, \'etale locally over $a\in A$, the morphism
$Pic_{\mathcal{S}_{U}/U}\ra A$ admits a section.  CLAIM 1A  is proved,  Part (i), and thus (ii), (iii) and (iv), follow.


 Finally, since $\ah_B$, $\beta_B$, and $\gam_B$ are smooth, 
 we have that their domains are nonsingular.
 By Lemmata \ref{nonempty} and \ref{t is smooth}, we have that the fibers of 
 $\ah_B$, $\beta_B$, and $\gam_B$ are integral, in particular connected.
 Since moreover their images are connected, we have that their domains must also be connected, thus integral. 
 \end{proof}

\begin{pr}[{\bf (Properness of Hodge-Hitchin over a base)}]\la{prpbz}$\;$

\n
Let $\csub$ be a smooth curve (\S\ref{rino}) over a Noetherian integral and normal  base $\base$ that is of finite type
over   an algebraically closed field of characteristic $p>0.$
Assume that $d=\ov{d}p$ is a multiple of $p$ and that ${\rm g.c.d.} (r,\ov{d})=1.$
The Hodge-Hitchin morphism $h_{Hod}$
 (\ref{mz2})  is proper, in fact projective.
\end{pr}
\begin{proof}
Since the Hodge-Hitchin morphism is quasi projective, it is enough to prove it is proper.
To this end, it is enough to  verify the hypotheses (1-4) in  the Properness Criterion \ref{ad hoc pr}, as  it has been done 
in the proof of
Proposition \ref{hh is proper}. The verification  is completely analogous.
\end{proof}

\subsection{Compactifications, vanishing cycles and specialization}\la{cova}$\;$

Recall that if a family  is not proper over a  Henselian DVR (or, more geometrically, over a smooth curve), then  the specialization morphism (\ref{rr12})  is not necessarily defined and, moreover, smoothness
of the family alone is not sufficient in general  to infer the vanishing  we prove next. Such issues have been tackled
over the complex numbers in \ci{de-2021}. The discussion \ci[\S5.1]{de-zh B}  shows that under favorable circumstances, we can apply the results in \ci{de-2021}, originally proved over the complex numbers,
to a situation  over an algebricaically closed field, and over a  DVR. Based on this, we state and prove the following

\begin{pr}\la{sp cptz}$\;$

\ben\item
Let things be as in  \S\ref{the mainz}: 
 $C/\field$, is a smooth curve (\S\ref{rino}) over an algebraically closed field $\field$,   $char (\field)=p>0,$  and ${\rm g.c.d.} (r,d=\ov{d}p)=1$. 
 Let  $\phi_\tau$ be the vanishing cycle functor  (\S\ref{rem vn}) associated with the morphism
 $\tau_{Hod}: M_{Hod} \to \Aa^1_\field$ (\ref{mz2})
  after base change  $S\to \Aa^1_\field$  from  the a strict Henselianization of $\Aa^1_\field$ at the origin.  We have the identity $\phi_\tau (\tau_* \oql)=0$ for the vanishing cycles (\S\ref{rem vn}).

\item
Let $\csub$ be a smooth curve (\S\ref{rino}) where $\base$ 
is (the spectrum of) a   strictly
Henselian DVR (\S\ref{rem vn}).
Assume ${\rm g.c.d.}(r,d)=1$ and, when the DVR is of mixed characteristic $(0, p>0)$,  also assume  that $p>r$.
The specialization morphism
\beq\la{r575}
\xymatrix{
H^* \left( M_{Dol} (C_s; r, d), \oql  \right)  \ar[r]^-{\rm sp} & 
H^* \left( M_{Dol} (C_{\ov{\eta}}; r, d), \oql \right)
}
\eeq
is defined, it is a cohomology ring isomorphism, and a filtered isomorphism for the perverse Leray filtrations induced by the
respective Hitchin morphisms (\ref{mhi}).
\een
\end{pr}
\begin{proof} According to the discussion \ci[\S5.1]{de-zh B}, we can apply \ci[Lm.\;4.3.3]{de-2021} (resp.\;\ci[Tm.\;4.4.2]{de-2021}) to the present situation (1) 
(resp.\;(2)), as long as the morphism $M_{Hod}(C/\field) \to \Aa^1_\field$ (resp.\;$M_{Dol}(C/B) \to B$)
is smooth and the moduli space universally corepresents the appropriate functor. 
The smoothness  has been  proved in Proposition \ref{rel smooth}.(3) (resp.\;\ref{rel smooth}.(2)), and,
in view of the fact that stability equals semistability in the coprime case, 
 the universal corepresentability  in the coprime case is due to A. Langer \ci[Tm. 1.1]{la-2014}. 
This implies the desired conclusion (1) (resp.\;(2)).
 \end{proof}

\begin{rmk}\la{d> 1}
If we replace the Dolbeault moduli spaces in Theorems \ref{vvooqq}, \ref{c cop} and \ref{c cop p} with the moduli space of stable $L$-twisted Higgs bundles of degree coprime to the rank, where $L$ is either the canonical bundle, or it satisfies $\text{deg }L> \text{deg }\omega_{C}$, then we still have the analogous conclusion as in Proposition \ref{sp cptz}.(2).
This is because the analogue of Proposition \ref{rel smooth}.(2) holds  by the coprimality condition,
with virtually the same proof.
\end{rmk}

\subsection{Second  proof of Theorem \ref{zmain}}\la{second proof}$\;$

In this section, we use Proposition \ref{sp cptz} to give a second and simpler proof of Theorem \ref{zmain}. In fact, this proof yields
an even stronger statement. On the other hand,  the proof of \ref{zmain} is more self-contained and, importantly, brings to the front
the isomorphism (\ref{to ex}), which plays a key role in the proof of   Theorem \ref{central fiber}, which  is key to 
proving the 
 $p$-Multiplicative periodicity with Frobenius twists Theorem \ref{int cons}, which in turn plays a repeated role henceforth.

Recall our assumptions \S\ref{the mainz}: 
 $C/\field$, $char (\field)=p>0,$  and ${\rm g.c.d.} (r,d=\ov{d}p)=1$. 
 
\begin{tm}[{\bf (Cohomological Simpson Correspondence $char(\field)=p>0$, II)}]\la{via pb} $\;$

\n
The inclusions $i_0: M_{Dol}\to M_{Hod}$ and $i_1:M_{dR} \to M_{Hod}$ induce
filtered  isomorphisms of  cohomology rings
\beq\la{rollo}
\xymatrix{
H^*(M_{Dol} (C;r,\ov{d}p)  ,\oql)
&
H^*(M_{Hod} (C;r,\ov{d}p)  ,\oql) \ar[l]_{i_0^*}^-\simeq \ar[r]_-\simeq^{i_1^*}
&
H^*(M_{dR} (C;r,\ov{d}p)  ,\oql)
}
\eeq
for the perverse Leray filtrations  associated with the Hitchin, the Hodge-Hitchin and the de Rham-Hitchin morphism, respectively.
\end{tm}
\begin{proof}
By virtue of the smoothness of $\tau_{Hod}$ (Theorem \ref{t is smooth}) and of the properness of the Hodge-Hitchin morphism
(Theorem \ref{hh is proper}),
 we can apply  Proposition \ref{sp cptz},  and we have   $\phi_\tau (\tau_* \oql)= 0$. 

Since $\w{\phi_\tau:}= \phi_\tau [1]$ is $t$-exact for the perverse $t$-structure, we have the identity  
\[
\w{\phi_\tau}  (\pcs^\bullet{(\tau_* \oql)}) =   \pcs^\bullet{(\w{\phi_\tau}(\tau_* \oql))}=0
\]
relating  perverse cohomology sheaves.
The local trivialization (\ref{prfr}) implies that  the restriction 
$\pcs^\bullet {(\tau_* \oql)}_{|{\Gm}_{,\field}} \simeq \ms{L}^\bullet [1]$, where $\ms{L}^\bullet$ is a suitably constant sheaf on ${\Gm}_{,\field}.$

By combining the two assertions of the previous paragraph with A. Beilisnon's description of perverse sheaves via the vanishing cycle functor (see
\ci[Prop. 3.1]{beilinson}, or \ci[Thm. 5.7.7]{bams},
for example), we see that the perverse cohomology sheaves $\pcs^\bullet {(\tau_* \oql)}$ are constant sheaves shifted by $[1]$.

A simple induction using the perverse truncation distinguished triangles, coupled with the fact that
$H^{\bullet \neq 0} (\Aa^1_\field, \oql)=0$, shows that the complex $\tau_* \oql$ splits as the direct sum
of its shifted perverse cohomology sheaves, and thus, because they are shifts of constant sheaves, as the direct sum  $\oplus_{i\ge0} R^i\tau_* \oql[-i]$ of its shifted direct image sheaves which, moreover,  are constant sheaves of some rank.

The unfiltered assertion (\ref{rollo}) follows.
For the filtered version we argue similarly, replacing $\tau_* \oql$ with 
the sequence of complexes ${{\rm pr}_2}_* \ptd{\bullet} {h_{Hod}}_* \oql$   (cf. (\ref{mz2})).
\end{proof}

\begin{rmk}\la{007}
We can also prove Theorem \ref{via pb}, without using Beilinson's glueing of perverse sheaves, as follows:

Since $\phi_{\tau}(\tau_*\oql)=0$, we have that $R^i\tau_*\oql$ is locally constant for each $i$. 
We also know that $R^i\tau_*\oql$ is constant over $\Gm$.
Therefore the local system $R^i\tau_*\oql$ is determined by a continuous representation $\pi_1(\mathbb{A}^1_k, 1)$ into $GL(H^i(M_{dR},\oql))$ of the \'etale fundamental group such that 
the composition with $\pi_1(\Gm,1)\to \pi_1(\mathbb{A}^1_k, 1)$ is trivial.
Since the morphism $\pi_1(\Gm,1)\to \pi_1(\mathbb{A}^1_k,1)$ is surjective \ci[\href{https://stacks.math.columbia.edu/tag/0BQI}{0BQI}]{stacks-project}, we have that the representation $\pi_1(\mathbb{A}^1_k,1)\to GL(H^i(M_{dR},\oql))$ is also trivial, so that $R^i\tau_*\oql$ is constant over $\mathbb{A}^1_k$. 
\end{rmk}

\begin{rmk}\la{mot}
If we disregard the filtrations, the  ring isomorphisms (\ref{rollo}) lift to Voevodsky motives:  one combines the following two results 
 \ci[Thm. B1,  Cor. B2, and the method of proof of Thm. 4.2]{ho-le} with the  setup and smoothness results of this paper.  
\end{rmk}

\subsection{Cohomology ring of Dolbeault moduli spaces for two distinct curves}\la{?m?}$\;$

The goal of this subsection section is to prove   Theorem \ref{vvooqq}, which, over the complex numbers,  is an immediate consequence of the Simpson correspondence,
for the two Dolbeault spaces have isomorphic Betti moduli spaces, to which they are canonically homeomorphic.

\begin{tm}
[{\bf (Different curves, same degree)}]
\la{vvooqq}
Let $C_i/\field$ be two smooth curves (\S\ref{rino}) over an algebraically closed field.
Assume that rank and degree are coprime ${\rm g.c.d.} (r,d)=1$ 
(we do not assume that $d$ is a multiple of $p$).
There is  a non canonical  isomorphism  of   cohomology rings which is a filtered isomorphism for the perverse Leray filtrations
stemming from the  respective Hitchin morphism 
\begin{equation}
\label{twocurves}
    \xymatrix{
    H^*(M_{Dol}(C_1;r,d), \oql)  \ar[r]^-\simeq_-{(*)} &
    H^*(M_{Dol}(C_2;r,d), \oql).
    }
\end{equation}
If, in addition,  the ground field is of characteristic $p>0$, and $d=\ov{d}p$ is an integer multiple of $p$,  then 
we have a commutative diagram of  isomorphisms of cohomology rings which are filtered isomorphisms
for the respective perverse Leray filtrations
\begin{equation}
\label{twocurves bis}
    \xymatrix{
    H^*(M_{Dol}(C_1; r, \ov{d}p), \oql)  \ar[r]^-\simeq_{(*)} \ar[d]^-\simeq   &
    H^*(M_{Dol}(C_2; r, \ov{d}p), \oql)   \ar[d]^-\simeq
    \\
    H^*(M_{dR}(C_1; r, \ov{d}p), \oql)  \ar@{.>}[r]^-\simeq &
    H^*(M_{dR}(C_2; r, \ov{d}p), \oql).
    }
\end{equation}
\end{tm}
\begin{proof}
The second statement  (\ref{twocurves bis}) follows easily from the first one (\ref{twocurves}) as follows:
we take the vertical  isomorphisms in (\ref{twocurves}) to be the  canonical ones of Theorem \ref{zmain};
we take $(*)$ to be the one in (\ref{twocurves}); we close the diagram in the evident fashion.

We now construct the isomorphism $(*)$ in (\ref{twocurves}).

Let $g$ be the genus of the curves $C_1, C_2.$ If $g=0$, then the Dolbeault moduli spaces in questions
are a single point for $r=1$ and empty for $r>1$ (\ci[\S7]{nitsure}) in either case, there is nothing left to prove.  
If $g=1,$ then we argue as in the forthcoming $g\geq 2$ case, by using the  irreducible moduli space of $g=1$ curves with level structure \cite[Cor.\;5.6]{dr73}.
We may thus assume that $g\geq 2.$

By the irreducibility assertion \cite[\S3]{dm} for the Hilbert scheme of tri-canonically embedded curves of genus $g\geq 2$,
we can find a projective and smooth  family $\csub$ of  genus $g$ curves, with $B$ a nonsingular connected curve
and with  two closed fibers $X_{b_i} \simeq C_i,$ for $b_i \in B,$ $i=1,2.$

We conclude by taking $(*)$ to be  (\ref{r575}) as in Proposition \ref{sp cptz}.(2) (triangulate
$b_1$ and $b_2$ through a geometric generic point of $B$),  which we can use in view of the 
smoothness assertion in Theorem \ref{rel smooth}.(2).
\end{proof}

\begin{rmk}\la{larb}
The conclusion (\ref{twocurves}) in Theorem \ref{vvooqq} holds, with the same proof,  in the set up
of Remark \ref{d> 1}. The key points are the properness of the Hitchin morphism in families \ci{fa, nitsure, si mo II}, and  the smoothness of the Dolbeault moduli space (the same proof as the one of Proposition 
 \ref{rel smooth}.(2) goes through).
\end{rmk}

\subsection{\texorpdfstring{$p$}{p}-Multiplicativity without Frobenius twist}$\;$

{\bf The well-known additive periodicity of Dolbeault moduli spaces.}
Let $C$ be  a connected nonsingular  projective curve  over an algebraically closed field $\field$.
For arbitrary degree rank $r$ and  $d \in \zed$,    there is a canonical isomorphisms of cohomology rings for every $n \in \zed$
\beq\la{gkgk4}
H^* \left( M_{Dol} \left(  C;r , {d} \right) \right) \simeq 
H^* \left( M_{Dol} \left(  C; r , {d} +rn  \right) \right).
\eeq
This   follows from the fact that that the choice of any degree $n$ line bundle
$L$ on $C$ induces, by the assignments $(E,\phi) \mapsto (E \otimes L, 1_L\otimes \phi)$  an isomorphism of Dolbeault moduli spaces
that commutes with the Hitchin morphisms, 
hence induces a filtered isomorphism  of cohomology rings as in (\ref{gkgk4}). Since $L$ can be made to vary in the connected $Pic^n(C)$, we have that
this latter isomorphism is independent of the choice of $L \in Pic^n(C).$

We have the following consequence of Theorems \ref{int cons} and    \ref{vvooqq} which came as a  surprise to us.
Note the very different nature  of (\ref{dfr}), i.e. its expressing a periodicity under multiplication of the degree (coprime to the rank)  by powers of  $p$,  when compared with   (\ref{gkgk4}),
which expresses a periodicity when adding multiples of the rank to the degree.
 
 The following result is concerned with the curve $C$ only, and  should be compared with Theorem \ref{int cons}
 which is concerned with a curve $C$ and with its Frobenius twist $\FC.$
 
\begin{tm}
[{\bf ($p$-Multiplicative periodicity without Frobenius twists)}]\la{e4}
Let $C/\field$ be a smooth curve (\S\ref{rino})  over an algebraically closed field $\field$ of characteristic $p >  0$.
Assume that
${\rm g.c.d.} (r,d)=1$ (we do not assume that $d$ is a multiple of $p$).

For every  $m \in \zed^{\geq 0}$, there is  a non canonical isomorphism of cohomology rings
\beq\la{dfr}
H^* \left( M_{Dol} \left( C; r , d\right)  \right) \simeq H^* \left( M_{Dol} \left( C; r,  d p^m \right)  \right).
\eeq
which is a  filtered isomorphism for the  perverse Leray filtrations associated with the  Hitchin morphism
$M_{Dol} (C) \to A(C).$
\end{tm}
\begin{proof}
Combine 
Theorems \ref{int cons} and    \ref{vvooqq},  this latter with  $C_1:=C$ and $C_2:= \FC$. 
\end{proof}

\subsection{Cohomology ring of Dolbeault moduli spaces for two distinct degrees}\la{co co}$\;$

In this section we prove Theorems \ref{c cop} and \ref{c cop p}.

\begin{tm}
[{\bf  (Same curve, different degrees; $char (\field)=0$)}]
\la{c cop}$\;$

\n
Let $C/\field$ be a smooth curve (\S\ref{rino})  over an algebraically closed field of characteristic zero.  
Fix the positive integer $r$ (the rank). Let $d,d'$ (the degrees) be any two integers coprime with
$r.$ There is a non-canonical ring isomorphism 
\beq\la{p11}
\xymatrix{
H^*(M_{Dol}(C;r,d), \oql)  & \cong & H^*(M_{Dol}(C;r,d'), \oql)
}
\eeq
which is a filtered isomorphism for the perverse Leray filtrations associated with the respective Hitchin morphisms.
\end{tm}
\begin{proof}

Let $a\in \zed$ be such that $da \equiv d' \mod r.$ By the Dirichlet Prime Number Theorem there are infinitely many prime congruent to $a$ modulo $r.$ Choose any such prime $p$ such that $p>r$ and $p \ne \ell$ ($\ell$ as  in $\oql$).

By the $r$-periodicity (\ref{gkgk4}) and the $p$-multiplicativity (\ref{dfr}), the 
statement of the theorem is true if  we replace the characteristic zero algebraically closed ground field, with any algebraically closed ground field of characteristic $p$.

By the Lefschetz Principle, we can replace the given ground field,  by any algebraically closed field of characteristic zero,
such as the forthcoming $\ov{\kappa (\alpha)}$. 
In view of the isomorphisms (\ref{twocurves}), we can also replace the given curve $C$  with any other curve  of the same genus over  $\ov{\kappa (\alpha)}$, such as the forthcoming $X_{\ov{\kappa (\alpha)}}$.

Let $A$ be the spectrum of a complete DVR 
of characteristic zero  with  algebraically closed residue field  $\field$ of characteristic $p.$ The content of this paragraph, namely
that curves in positive characteristic can be lifted to characteristic zero, is standard and well-known.
For example, see \ci[Prop. 2.1]{ob}; see also
\href{https://amathew.wordpress.com/tag/lifting-to-characteristic-zero/}{this post}  (Def. 4 and Thm. 5), and also 
\href{https://amathew.wordpress.com/2011/06/18/lifting-smooth-curves-to-characteristic-zero/#more-2687}{its continuation}. 
There is a smooth curve $X/A$, with closed special  fiber $X_a$
any pre-chosen  integral nonsingular projective curve  of genus $g$  over $\kappa (a)$, and with generic geometric fiber   
$X_{\ov{\kappa (\alpha)}}$
a curve of the same kind, but over the algebraically closed field $\ov{\kappa (\alpha)}$ given by any chosen algebraic closure of the residue 
field $\kappa (\alpha)$ of the generic point  $\alpha \in A.$

By combining the characteristic $p$ version of (\ref{p11}) with Proposition \ref{sp cptz}.(2), we get the following chain of cohomology ring isomorphisms, which are filtered isomorphisms for the respective perverse Leray filtrations
(we drop the rank $r$)
\beq\la{4400}
  H^*(M_{Dol}(X_{\ov{\kappa (\alpha)}}; d)) \cong H^*(M_{Dol}(X_a; d)) \cong 
H^*(M_{Dol}(X_a; d')) \cong  H^*(M_{Dol}(X_{\ov{\kappa (\alpha)}}; d')).
\eeq

The theorem is thus proved.
\end{proof}

Note that in the proof of Theorem \ref{c cop} above, one can  avoid using Proposition \ref{sp cptz}.(2) by spreading out $C$, instead of lifting a chosen $X_a$.
However,  we use the lifting of $X_a$ and Proposition \ref{sp cptz}.(2)  in the proof of Theorem \ref{c cop p} below.

\begin{tm}
[{\bf  (Same curve, different degrees; $char (\field)=p>r$)}]
\la{c cop p}$\;$

\n
Let $(r,d,d')$ be such that ${\rm g.c.d.} (r,d)=g.c.d.(r,d')=1$.
Let $C/\field$ be a smooth curve (\S\ref{rino})  over an algebraically closed field $\field$ of characteristic $p>r$. 
There is a non-canonical ring isomorphism
\beq\la{p12}
\xymatrix{
H^*(M_{Dol}(C,r,d), \oql)  & \cong & H^*(M_{Dol}(C,r,d'), \oql)
}
\eeq
which is a filtered isomorphism for the perverse Leray filtrations associated with the respective Hitchin morphisms.
\end{tm}
\begin{proof}
Let $X/A$ be a lift of $C$ to characteristic zero as in the proof of Theorem (\ref{c cop}).
The desired conclusion in positive characteristic $p$ follows 
by combining  the analogous result  (\ref{p11}) in characteristic zero, with the specialization isomorphism  (\ref{r575}).
\end{proof}

Note that Theorem \ref{c cop p} does not follow immediately by combining 
the $p$-multiplicativity (\ref{p-per-2}) with the elementary periodicity (\ref{gkgk4}) with respect to the rank. For example, take $p=3, r=13, d'=1, d=15$.

\begin{rmk}\la{491}
One can combine the results of Theorem \ref{vvooqq}, with the ones of Theorems \ref{c cop}, \ref{c cop p},
and obtain  the evident  ``different curves, different degrees" version (omitted).
\end{rmk}

\begin{rmk}[{\bf (Earlier results)}]\la{wiskn}$\;$

\ben
\item
Point counts over finite fields, coupled with smoothness and purity arguments, give an equality of  Betti numbers for the two sides of (\ref{p11})
and (\ref{p12})
over an algebraically closed ground field;  see \ci{gwz, me, mo-sc, sc}.
While such methods imply the existence of an additive isomorphism preserving the perverse filtration, they   do not seem to yield information on cup products.
 
 \item
Let the ground field be the complex numbers.
If we replace $M_{Dol}$ with the Betti moduli space $M_B$, then a well-known Galois-conjugation method yields
a canonical isomorphism  of cohomology rings analogous to (\ref{p11}). By the Non Abelian  Hodge Theory for 
${\rm g.c.d.}(r,d)=1$
over the complex numbers  (\ci{ha-th-PLMS}),
we have cohomology ring  isomorphisms $H^*(M_{B}) \simeq H^*(M_{Dol})$, so that we obtain a canonical cohomology ring isomorphism
as in (\ref{p11}), but different from it. We are unaware of 
an evident reason why this canonical isomorphism should be compatible with the perverse filtration, the way   (\ref{p11}) is. Added in revision: this issue is settled in the positive in \ci{dmsz}.

\item
Over a ground field of positive characteristic, given the lack of a Betti moduli space counterpart,  the  existence of a multiplicative  (\ref{p12})  is new, and   so  is its compatibility with the perverse filtrations associated with the Hitchin morphisms.
\een
\end{rmk}

\end{document}